\documentclass[notitlepage, 12pt]{article}
\usepackage[nottoc]{tocbibind}
\usepackage{amsmath}
\usepackage{amssymb}
\usepackage{amsthm}
\usepackage{setspace}
\usepackage{comment}
\usepackage{xy}
\input xy
\xyoption{all}
\usepackage{enumitem}
\usepackage{tabularx}
\usepackage{hyperref}

\DeclareMathOperator{\Ker}{Ker}

\DeclareMathOperator{\End}{End}

\DeclareMathOperator{\Span}{Span}

\DeclareMathOperator{\Inf}{Inf}
\DeclareMathOperator{\mult}{mult}

\DeclareMathOperator{\supp}{supp}

\theoremstyle{plain}
\newtheorem{Theorem}{Theorem}[section]
\newtheorem{Proposition}[Theorem]{Proposition}
\newtheorem{Corollary}[Theorem]{Corollary}
\newtheorem{Lemma}[Theorem]{Lemma}

\theoremstyle{definition}
\newtheorem{Definition}[Theorem]{Definition}
\newtheorem{Remark}[Theorem]{Remark}
\newtheorem{Example}[Theorem]{Example}
\newtheorem{Construction}[Theorem]{Construction}
\newtheorem{Notation}[Theorem]{Notation}

\newcommand{\arr}{\ar@{>>}}
\newcommand{\dotar}{\ar@{.>}}
\newcommand{\dotarr}{\ar@{.>>}}

\newcommand{\id}{\mathrm{id}}
\newcommand{\pr}{\mathrm{pr}}

\newcommand{\N}{\mathbb{N}}

\newcommand{\PP}{\mathcal{P}}

\newcommand{\st}{|\,}

\newcommand{\pihat}{{\hat{\pi}}}
\newcommand{\epshat}{{\hat{\eps}}}
\newcommand{\etahat}{{\hat{\eta}}}

\newcommand{\Ehat}{{\hat{E}}}

\newcommand{\Ghat}{{\hat{G}}}
\newcommand{\Hhat}{{\hat{H}}}

\newcommand{\Ebar}{{\bar{E}}}

\newcommand{\Gbar}{{\bar{G}}}
\newcommand{\Hbar}{{\bar{H}}}

\newcommand{\etabar}{{\bar{\eta}}}
\newcommand{\thetabar}{{\bar{\theta}}}

\newcommand{\rhobar}{{\bar{\rho}}}

\def\phi{\varphi}

\newcommand{\calD}{\mathcal{D}}
\newcommand{\calE}{\mathcal{E}}

\newcommand{\calH}{\mathcal{H}}

\newcommand{\calM}{\mathcal{M}}

\newcommand{\eps}{\varepsilon}

\newcommand{\na}{\mathrm{na}}
\newcommand{\ab}{\mathrm{ab}}

\renewcommand{\iff}{\Leftrightarrow}
\renewcommand{\implies}{\Rightarrow}
\newcommand{\isom}{\cong}
\newcommand{\normal}{\triangleleft}

\newcommand{\pwr}[2]{\,{}^{#1}\kern-1pt {#2}\,}

\newcommand{\fprod}[2][G]{\underset{#2\phantom{\, {#1}}}{{\prod}_{#1}}}

\newcommand{\PPP}{\PP/\kern-5pt\sim}

\newcommand{\onto}{\twoheadrightarrow}
\newcommand{\mle}{\preccurlyeq}

\makeatletter
\def\moverlay{\mathpalette\mov@rlay}
\def\mov@rlay#1#2{\leavevmode\vtop{%
   \baselineskip\z@skip \lineskiplimit-\maxdimen
   \ialign{\hfil$\m@th#1##$\hfil\cr#2\crcr}}}
\newcommand{\charfusion}[3][\mathord]{
    #1{\ifx#1\mathop\vphantom{#2}\fi
        \mathpalette\mov@rlay{#2\cr#3}
      }
    \ifx#1\mathop\expandafter\displaylimits\fi}
\makeatother

\newcommand{\dotcup}{\charfusion[\mathbin]{\cup}{\cdot}}
\newcommand{\bigdotcup}{\charfusion[\mathop]{\bigcup}{\cdot}}

\newcommand{\subdemoinfo}[2]{\smallskip\noindent\textbf{#1:}
\emph{#2}\quad}

\begin{document}

\author{Dan Haran}
\title{On the uniqueness of the smallest embedding cover
}

\maketitle
\thispagestyle{empty}

\begin{abstract}
We prove the uniqueness of
a smallest embedding cover of a profinite group.
\end{abstract}

\clearpage
\pagenumbering{arabic}

\newpage
\section*{Introduction}
The purpose of this note is to
show the uniqueness of the smallest embedding cover
(\cite[Definition~27.4.3]{FJ};
Definition~\ref{definition; smallest embedding cover} here).
Chatzidakis \cite{Ch} proves this using model theoretic methods.
It is an open problem
(\cite[Problem 36.2.25]{FJ}
to prove the uniqueness
using only group theoretic methods.

Before we briefly explain 
how this is done here,
we introduce some notation for the whole paper
(which will also serve us for this explanation).

We use the word `family' for indexed family.

We use $\onto$ to denote an epimorphism of profinite groups.
An epimorphism $\eta \colon H \onto G$
is sometimes called
\textbf{cover} (of $G$)
and sometimes 
\textbf{extension} (of $G$).
Another epimorphism
$\pi \colon E \onto G$
is \textbf{dominated by $\eta$}
(or \textbf{$\eta$ dominates $\pi$}),
denoted $\pi \mle \eta$,
if there is
$\psi \colon H \onto E$
such that
$\eta = \pi \circ \psi$;
if such $\psi$ is an isomorphism of profinite groups,
we say that
$\pi$
is \textbf{isomorphic to $\eta$},
denoted $\pi \isom_G \eta$.

Now,
Definition~\ref{definition; smallest embedding cover}
defines an embedding cover,
and then it defines an I-cover
as a cover dominated by an embedding cover;
this is needed to define
a smallest embedding cover.
However, this definition of an I-cover
is non-intrinsic, and therefore hard to work with.
We define intrinsically
(Definition~\ref{basic cover})
\textbf{basic covers}
that are, to begin with,
simple examples of I-covers,
but later we show that
they play a central role in the characterization 
of arbitrary I-covers.

We heavily rely on our recent work
\cite{fund}
in which we characterize general epimorphisms of profinite groups.
More precisely,
\cite{fund}
presents an epimorphism $H \onto G$
as the inverse limit of a unique sequence of 
\textbf{fundaments}
\begin{equation*}
\xymatrix{
H \arr[r] &
\quad \cdots  \arr[r]^{\pi_3}
& G_2 \arr[r]^{\pi_2}
& G_1 \arr[r]^{\pi_1}
& G_0\rlap{ $ = G$}
}
\end{equation*}
that are (generalized) fiber products of indecomposable epimorphisms.
This presentation allows us to reduce the study of I-covers
to fundamental I-covers.
And it turns out
(Corollary~\ref{finite quotients of sec-k})
that fundamental I-covers with finite kernels
are precisesly basic covers.
This allows us to give a precise characterization
(Theorem~\ref{complete characterization})
of fundamental I-covers
in terms of the invariants introduced in
\cite{fund},
and thus to prove 
the desired uniqueness.

\section{Preliminaries}

An epimorphism
$\eta \colon H \onto G$
is \textbf{indecomposable} if
it is not an isomorphism
and
whenever $\eta$ is the composition of two epimorphisms
$\eta = \eta_2 \circ \eta_1$,
then either $\eta_1$ or $\eta_2$ is an isomorphism.

Consider a cartesian square
(\cite[Definition 25.2.2]{FJ})
of epimorphisms
\begin{equation}\label{basic}
\xymatrix{
H \arr[r]_\eta \arr[d]^\beta & G \arr[d]^\phi \\
B \arr[r]^\alpha & A \\
}
\end{equation}
(thus,
up to an isomorphism,
$H$ is the fiber product
$B \times_{A} G$ of $\alpha$ and $\phi$,
with coordinate projections $\beta$ and $\eta$.)
By \cite[Corollary 2.2]{fund},
if \eqref{basic} is cartesian, then
$\eta$ is indecomposable if and only if $\alpha$ is indecomposable.

A commutative square \eqref{basic} is cartesian
if and only if
$\Ker \beta \cap \Ker \eta = 1$ and
$(\Ker \beta)(\Ker \eta) = \Ker \phi \circ \eta$.
We call \eqref{basic} a \textbf{semi-cartesian square}
(\cite[Definition 2.4]{fund})
if the latter of these two conditions is satisfied.

A cartesian square \eqref{basic} is \textbf{compact}
(\cite[Definition 2.9]{fund})
if
no proper closed subgroup of $H$ is mapped simultaneously
by $\beta$ onto $B$
and by $\eta$ onto $G$.
Equivalently,
if $E$ is a profinite group
and $\beta' \colon E \onto B$ and $\eta' \colon E \onto G$
satisfy
$\alpha \circ \beta' = \phi \circ \eta'$,
then the unique homomorphism
$\psi \colon E \to H$ such that
$\beta \circ \psi = \beta'$ and $\eta \circ \psi = \eta'$
is surjective.

\begin{Proposition}
[{\cite[Proposition 2.12]{fund}}]
\label{indecomposable cartesian square}
A cartesian square \eqref{basic}
with $\alpha$ indecomposable
is compact
if and only if
$\phi$ does not dominate $\alpha$.
\end{Proposition}

\begin{Remark}
[{\cite[Lemma 2.3]{FH}}]
\label{square trivialities}
In the following commutative diagram
of epimorphisms with three squares
(the left one, the right one, and \eqref{basic}):
\begin{equation}\label{two squares}
\xymatrix@=24pt{
H
\arr[r]_{\zeta} \arr[d]^{\beta}
\arr@/^1pc/[rr]^{\eta}
& H' \arr[r]_{\eta'} \arr[d]^{\beta'}
& G \arr[d]^{\phi} 
\\
B
\arr@/_1pc/[rr]_{\alpha}
\arr[r]^{\gamma} & B' \arr[r]^{\alpha'} & A \\
}
\end{equation}
\begin{itemize}
\item[(a)]
If two of the squares are cartesian,
then so is the third one.
\item[(b)]
If all the three squares are cartesian,
then 
\eqref{basic} is compact
if and only if
both the left and right square are compact.
\end{itemize}
Consider a cartesian square \eqref{basic}.
\begin{itemize}[resume]
\item[(c)]
If there exist
either
$\gamma \colon B \onto B'$
and
$\alpha' \colon B' \onto A$
such that
$\alpha' \circ \gamma = \alpha$,
or
$\zeta \colon H \onto H'$
and
$\eta' \colon H' \onto G$
such that
$\eta' \circ \zeta = \eta$,
then there exists a commutative diagram \eqref{two squares}
with three cartesian squares.
\end{itemize}
\end{Remark}

\begin{Lemma}
[{\cite[Lemma 2.8(b)]{fund}}]
\label{expand}
Consider two commutative diagrams
\newline
\begin{subequations}
\begin{tabularx}{\textwidth}{Xp{.1cm}X}
\begin{equation}\label{semi1}
\xymatrix@=24pt{
H' \arr[r]_{\eta'} \arr[d]^{\beta'} & G \arr[d]^\phi \\
B \arr[r]^\alpha & A \\
}
\end{equation}
& &
\begin{equation}\label{semi2}
\xymatrix@=24pt{
H \arr[r]_\eta \arr[d]^\beta & G \arr[d]^\phi \\
B \arr[r]^\alpha & A \\
}
\end{equation}
\end{tabularx}
\end{subequations}
and let
$\psi \colon H' \to H$
be a homomorphism
such that
$\beta' = \beta \circ \psi$
and
$\eta' = \eta \circ \psi$.
If
\eqref{semi1} is semi-cartesian
and
\eqref{semi2} is cartesian,
then $\psi$ is surjective.
\end{Lemma}

Let
$\calH = (\eta_i\colon H_i \onto G \st i \in I)$
be a family of epimorphisms.
Its
\textbf{fiber product}
is the profinite group
\begin{equation}\label{concrete}
\fprod{i \in I} H_i =
\Big\{
(h_i)_{i \in I} \in \prod_{i \in I} H_i
\st
\eta_i(h_i) = \eta_j(h_j)
\textnormal{ for all } i,j \in I
\Big\}
\end{equation}
together with the following epimorphisms:
the projection $\pr_{I,j}$ on the $j$-th coordinate,
$\pr_{I,j}\colon \fprod{i \in I} H_i \onto H_j$,
for every $j \in I$,
and
${\eta_I = \eta_j\circ \pr_{I,j} \colon \fprod{i \in I} H_i \to G}$
(independent of $j \in I$).
More generally,
for $J \subseteq I$ we denote by
$\pr_{I,J}$ the projection
$\fprod{i \in I} H_i \onto \fprod{i \in J} H_i$
on the coordinates in $J$
given by
$(h_i)_{i \in I} \mapsto (h_i)_{i \in J}$.

A fiber product
$\fprod{i \in I} H_i$
is \textbf{compact}
if
there is no proper subgroup $H$ of 
$\fprod{i \in I} H_i$
such that
$\pr_{I,j}(H) = H_j$
for every $j \in I$.
Obviously, if 
$\fprod{i \in I} H_i$
is compact,
and there is a family of epimorphisms
$(p_i \colon H \onto H_i)_{i \in I}$,
then the induced map
$p \colon H \to \fprod{i \in I} H_i$
is surjective,
because the subgroup $p(H)$ of
$\fprod{i \in I} H_i$
cannot be proper.

\cite[Section 3]{fund}
gives some elementary properties
of fiber products;
while 
\cite[Section 6]{fund}
brings more advanced results,
such as
Proposition~\ref{indecomposable quotients of a fiber product pi is id}
below.

\begin{Remark}\label{about abelian}
Fix a finite simple $G$-module $A$.
Then $F_A = \End_G(A)$ is a finite field
(\cite[Remark 4.3(a)]{fund})
and 
$H^2(G,A)$
is a vector space over $F_A$
(\cite[Remark 5.4(a)]{fund}).

Elements of
$H^2(G,A)$
represent extensions of $G$ by $A$.
Then $x,x' \in H^2(G,A)$
represent isomorphic extensions
if and only if
there is $a \in F_A^\times$
such that
$x' = a x$
(and $0$ represents the split extension).
Thus the set of isomorphism classes of extensions of $G$ by $A$ is
the projective space
$H^2(G,A)/F_A^\times$.
By abuse of notation
we refer to extensions as elements of
$H^2(G,A)$;
our assertions will not depend on the choice
of representatives modulo $F_A^\times$.
\end{Remark}

\begin{Proposition}[{\cite[Theorem~6.7]{fund}}]
\label{indecomposable quotients of a fiber product pi is id}
Let
$\eta_I \colon \fprod{i \in I} H_i \onto G$
be the fiber product
of a family
$(\eta_i \colon H_i \onto G)_{i \in I}$
of indecomposable extensions
of $G$.
Let
$\zeta \colon E \onto G$ be indecomposable.
Let $C =\Ker \zeta$;
if $C$ is non-abelian,
let
$I_\zeta = \{i \in I \st \eta_i \isom_G \zeta\}$,
and
if $C$ is abelian,
let
$I_C = \{i \in I \st \Ker \eta_i \isom_G C\}$.
%
%
Then $\zeta \mle \eta_I$
if and only if
exactly one of the following two conditions holds:
\begin{itemize}
\item[(a)]
$C$ is non-abelian
and
$I_\zeta \ne \emptyset$.
\item[(b)]
$C$ is abelian,
$I_C \ne \emptyset$,
and $\zeta$ is
a nontrivial linear combination
of $(\eta_i)_{i \in I_C}$
over $F_C = \End_G(C)$.
\end{itemize}
Moreover, if
$\eps \colon \fprod{i \in I} H_i \onto E$
is an epimorphism
such that
$\zeta \circ \eps = \eta_I$,
then
\begin{itemize}[resume]
\item[(a')]
if $C$ is non-abelian,
there is a unique $j \in I$
such that $\Ker \eps = \Ker \pr_{I,j}$;
this $j$ is in $I_\zeta$ and 
there is an isomorphism $\eps_j \colon H_j \to E$
such that
$\eps_j \circ \pr_{I,j}= \eps$
and
$\zeta \circ \eps_j = \eta_j$.

\item[(b')]
if $C$ is abelian,
there is
$\eps' \colon \fprod{i \in I_C} H_i \onto E$
such that
$\eps = \eps' \circ \pr_{I,I_C}$
and
$\zeta \circ \eps' = \eta_{I_C}$.
\end{itemize}
\end{Proposition}

\begin{Definition}\label{fundament}
Let $\pi \colon H \onto G$ be an epimorphism of profinite groups.
The \textbf{fundament kernel} of $\pi$,
denoted $\calM(\pi)$,
is the intersection of all
$N \normal H$
contained in $\Ker \pi$
such that
the map $H/N \onto G$
induced from $\pi$
is indecomposable.
Clearly $\calM(\pi) \normal H$ and $\calM(\pi) \le \Ker \pi$.
An epimorphism
$\bar\pi \colon \Hbar \onto G$
is the \textbf{fundament} of $\pi$,
if it is isomorphic to the quotient map
$\bar\pi \colon H/\calM(\pi) \onto G$,
that is,
if there is
$\rho \colon H \onto \Hbar$
such that
$\bar\pi \circ \rho = \pi$
and
$\Ker \rho = \calM(\pi)$;
we then also say that
$\bar\pi$ is
\textbf{the fundament of $\pi$ by $\rho$}.
We say that $\pi$ is \textbf{fundamental},
if $\bar\pi = \pi$, that is,
if $\calM(\pi) = 1$.
Notice that the fundament of an epimorphism
is fundamental.
\end{Definition}

\begin{Lemma}[{\cite[Corollary 8.6]{fund}}]
\label{fiber product is fundamental}
A cover of $G$ is fundamental
if and only if it is isomorphic
to the fiber product
of a family of indecomposable extensions of $G$.
\end{Lemma}

\begin{Notation}\label{Lambda}
For a profinite group $G$
let
\begin{itemize}
\item[(a)]
$\Lambda_\ab(G)$ be the set of isomorphism classes of
finite simple $G$-modules;
\item[(b)]
$\Lambda_\na(G)$
be a set of representatives of the isomorphism classes 
of indecomposable epimorphisms onto $G$ 
with non-abelian kernel;
\item[(c)]
$\Lambda(G) = \Lambda_\ab(G) \dotcup \Lambda_\na(G)$.
\end{itemize}
For every $C \in \Lambda_\ab(G)$
let $\sigma_C^{(0)}$ be the split extension of $G$ by $C$
and let $F_C = \End_G(C)$.

Furthermore,
let $(v_i)_{i \in I}$
be a family of vectors in a vector space $V$ over a field $F$.
Denoting by $F^I$ the vector space with basis $I$,
let $T \colon F^I \to V$
be the unique linear map that maps $i$ onto $v_i$,
for all $i \in I$.
Then 
$\dim \Ker T$ is the \textbf{relation dimension}
of $(v_i)_{i \in I}$.
\end{Notation}

\begin{Theorem}[{\cite[Theorem 8.7]{fund}}]
\label{fundamental mult supp}
Let $\pi \colon H \onto G$ be fundamental.
Then for every
$\lambda \in \Lambda(G)$
there is a unique cardinality
$\mult_\lambda(\pi)$
and for every $C \in \Lambda_\ab(G)$
there is a unique $F_C$-subspace $\supp_C(\pi)$ of $H^2(G,C)$
with the following property:

Let
$\eta_I \colon \fprod{i \in I} H_i \onto G$
be the fiber product
of a family
of indecomposable extensions
$(\eta_i \colon H_i \onto G)_{i \in I}$
of $G$.
Denote
$I_\eta =\{i \in I \st \eta_i \isom_G \eta\}$,
for every $\eta \in \Lambda_\na(G)$,
and
$I_C =\{i \in I \st \Ker \eta_i \isom_G C\}$,
for every $C \in \Lambda_\ab(G)$.
Then $\pi$ is isomorphic to $\eta_I$
if and only if
\begin{itemize}
\item[(a)]
$|I_\eta| = \mult_\eta(\pi)$,
for every $\eta \in \Lambda_\na(G)$,
\item[(b)]
$\supp_C(\pi) = \Span(\eta_i)_{i \in I_C}$,
for every $C \in \Lambda_\ab(G)$.

\item[(c)]
$\mult_C(\pi)$
is the relation dimension of
$(\eta_i)_{i \in I_C}$,
for every $C \in \Lambda_\ab(G)$.
\end{itemize}
\end{Theorem}

\begin{Corollary}[{\cite[Corollary 8.6]{fund}}]
An epimorphism
$\pi \colon H \onto G$
is fundamental
if and only if
it is
the fiber product of a family of indecomposable covers of $G$.
\end{Corollary}

\begin{Theorem}[{\cite[Corollary 8.10]{fund}}]
\label{compare}
Let $\tau \colon H \onto G$
and
$\tau' \colon H' \onto G$
be fundamental.
Then:
\begin{itemize}
\item[(i)]
$\tau' \mle \tau$
if and only if
\begin{itemize}
\item[(i1)]
$\mult_\lambda(\tau') \le \mult_\lambda(\tau)$,
for every $\lambda \in \Lambda(G)$;
and
\item[(i2)]
$\supp_C(\tau') \subseteq \supp_C(\tau)$,
for every $C \in \Lambda_\ab(G)$.
\end{itemize}

\item[(ii)]
$\tau' \isom_G \tau$
if and only if
\begin{itemize}
\item[(ii1)]
$\mult_\lambda(\tau') = \mult_\lambda(\tau)$,
for every $\lambda \in \Lambda(G)$;
and
\item[(ii2)]
$\supp_C(\tau') = \supp_C(\tau)$,
for every $C \in \Lambda_\ab(G)$.
\end{itemize}
\end{itemize}
\end{Theorem}

From Theorem~\ref{compare} one deduces:

\begin{Corollary}
[{\cite[Corollary 8.11]{fund}}]
\label{dominates multiple of one}
Let $\pi \colon H \onto G$ be fundamental.
Let $\eta$ be an indecomposable cover of $G$
and $\kappa \ge 0$ a cardinality.
Let $\eta^{(\kappa)}$ be the fiber product of $\kappa$ copies of $\eta$.
Put $C = \Ker \eta$.
\begin{itemize}
\item[(a)]
If $C$ is not abelian, then
$\eta^{(\kappa)} \mle \pi$ 
if and only if
$\kappa \le \mult_\eta(\pi)$.

\item[(b)]
If $C \in \Lambda_\ab(G)$, then
$\eta^{(\kappa)} \mle \pi$ 
if and only if
\begin{equation*}
\begin{cases}
\kappa \le 
\mult_C(\pi)&
\textnormal{if $\eta \in \supp_C(\pi)$ splits}
\\
\kappa \le 
\mult_C(\pi) + 1 &
\textnormal{if $\eta \in \supp_C(\pi)$ does not split}
\\
\kappa = 
0 &
\textnormal{if $\eta \notin \supp_C(\pi)$}
\end{cases}
\end{equation*}
\end{itemize}
\end{Corollary}

\begin{Remark}\label{hereditary}
Let $\pi \colon H \onto G$
and $\pi' \colon H' \onto G$
be two epimorphisms.
Let $\pi_1$ and $\pi_1'$ be their fundaments.
By \cite[Theorem~9.4(a)]{fund},
$\pi' \mle \pi \implies \pi_1' \mle \pi_1$.
In particular,
if $\pi'$ is fundamental,
then $\pi'_1 = \pi'$,
so 
$\pi' \mle \pi \iff \pi' \mle \pi_1$.
\end{Remark}

\begin{Lemma}\label{quotient of fprod not I}
Let 
$(\eta_i \colon H_i \onto G)_{i \in I}$
be an infinite family
$(\eta_i)_{i \in I}$
of isomorphic indecomposable epimorphisms.
Let $k \in I$ and let 
$J = I \smallsetminus \{k\}$.
Then
$\pr_{I,J} \colon \fprod{i \in I}H_i \onto \fprod{i \in J}H_i$
is indecomposable
and there is no
compact cartesian square
\begin{equation}\label{compact with pr}
\xymatrix{
\fprod{i \in I}H_i
\arr[r]^{\pr_{I,J}} \arr[d]^{\beta}
& \fprod{i \in J}H_i
\arr[d]^{\phi}
\\
B \arr[r]^{\alpha} & A
}
\end{equation}
with
$A, B$ finite.
\end{Lemma}

\begin{proof}
The fiber products
$\eta_I \colon \fprod{i \in I}H_i \onto G$
and
$\eta_J \colon \fprod{i \in J}H_i \onto G$
are isomorphic,
because $|I| = |J|$.

By \cite[Remark 3.2(e)]{fund}
the following is a cartesian square
\begin{equation}\label{simple partition}
\xymatrix@=30pt{
\fprod{i \in I}H_i \ar[r]_{\pr_{I,J}} \ar[d]^{\pr_{I,k}}
& \fprod{i \in J}H_i \ar[d]^{\eta_J}\\
H_k \ar[r]_{\eta_k} & G \rlap{,} \\
}
\end{equation}
hence, as $\eta_k$ is indecomposable, so is $\pr_{I,J}$.

Suppose there is a compact cartesian square
\eqref{compact with pr}
with $A, B$ finite.
As $\fprod{i \in I}H_i$ is the inverse limit of
$\fprod{i \in I'}H_i$,
where $I'$ runs through the finite subsets of $I$,
\ $\beta$ factors through $\pr_{I,I'} \colon \fprod{i \in I}H_i \onto
\fprod{i \in I'} H_i$,
for a sufficiently large $I'$.
Similarly
$\phi$ factors through
$\pr_{J,J'} \colon \fprod{i \in J}H_i \onto \fprod{i \in J'} H_i$,
for a sufficiently large finite  subset $J'$ of $J$.
We may assume that
$k \in I'$ and 
$J' = I' \smallsetminus \{k\}$.
Then the above factorizations give a commutative diagram
\begin{equation}
\xymatrix@=24pt{
\fprod{i \in I}H_i \arr[r]_{\pr_{I,J}} \arr[d]^{\pr_{I,I'}}
\arr@/_2pc/[dd]_{\beta}
& \fprod{i \in J}H_i \arr[d]_{\pr_{J,J'}} \arr@/^2pc/[dd]^{\phi}
\\
\fprod{i \in I'} H_i \arr[r]^{\pr_{I',J'}} \arr[d]^{\beta'}
& \fprod{i \in J'} H_i \arr[d]_{\phi'}
\\
B \arr[r]^{\alpha} & A
}
\end{equation}
in which the upper square is cartesian.
By Remark~\ref{square trivialities}(a)
both squares are cartesian;
by Remark~\ref{square trivialities}(b),
both squares are compact.

Let $j \in J \smallsetminus J'$;
then $j \in I \smallsetminus I'$.
By assumption there is an isomorphism
$\theta \colon H_j \to H_k$
such that
$\eta_k \circ \theta = \eta_j$.
For every $i \in I' = J' \dotcup \{k\}$
define
$\gamma_i \colon \fprod{i \in J}H_i \onto H_i$
by
$
\gamma_i = 
\begin{cases}
\pr_{J,i} & \textnormal{if $i \in J'$}\\
\theta \circ \pr_{J,j} & \textnormal{if $i = k$}
\end{cases}
$.
Then $(\gamma_i)_{i \in J}$ define a homomorphism
$\gamma \colon \fprod{i \in J}H_i \to \fprod{i \in I'} H_i$
such that
$\pr_{I',J'} \circ \gamma = \pr_{J,J'}$.
It is easy to see that
$\gamma$ is surjective.
%
%
This is a contradiction to the compactness
of the upper square,
by Proposition~\ref{indecomposable cartesian square}.
\end{proof}

\section{Embedding covers and I-covers}\label{emb cover}

\begin{Definition}\label{definition; smallest embedding cover}
(\cite[Definition 24.4.3]{FJ})
Let $G$ and $E$ be profinite groups.
\item [(a)]
$E$ has the \textbf{embedding property}
if for every epimorphism
$\alpha \colon B \onto A$
of finite quotients of $E$
and every $\phi \colon E \onto A$
we have
$\alpha \mle \phi$.
\item [(b)]
An epimorphism $\eps \colon E \onto G$ is an
\textbf{embedding cover} of $G$
if $E$ has the embedding property.
\item [(c)]
An epimorphism $\eta \colon H \onto G$ is an
\textbf{I-cover} of $G$
if every embedding cover of $G$ dominates $\eta$.
\item [(d)]
A\footnote{
The ungrammatical use of the indefinite article 
is justified by our intention 
not to assume the uniqueness of this cover,
precisely because we want to prove it.
}
\textbf{smallest embedding cover} of $G$
is an embedding cover which is also an I-cover.

A smallest embedding cover of a profinite group exists
(\cite[Proposition 27.4.5]{FJ},
cf. Construction~\ref{sec transfinite} below).
It is unique, up to isomorphism, by \cite{Ch}.
The goal of this paper
is to prove this group-theoretically.

\noindent
(e)
Let
$\calE(G)$, resp. $\calE_f(G)$,
be the sets
of isomorphism types of quotients,
resp.~of finite quotients,
of a smallest embedding cover $E$ of $G$.
These sets are uniquely determined by $G$,
by Remark~\ref{EV trivialities}(a) below.
\end{Definition}

\begin{Remark}\label{EV trivialities}
(a)
If $E \onto G$ and $E' \onto G$ are
smallest embedding covers,
then $E, E'$ are quotients of each other.

(b)
Let
$
\xymatrix{
  G_2 \arr[r]^{\pi_2}
& G_1 \arr[r]^{\pi_1}
& G
}
$ be epimorphisms.
If $\pi_1 \circ \pi_2$ is an I-cover,
then so is $\pi_1$.

(c)
The composition
of I-covers is an I-cover.
More generally,
let
\begin{equation*}
\xymatrix{
\cdots \arr[r]^{\pi_3}
& G_2 \arr[r]^{\pi_2}
& G_1 \arr[r]^{\pi_1}
& G_0\rlap{ $ = G$}
}
\end{equation*}
be a sequence of I-covers.
Then its inverse limit $\pi \colon H \onto G$
is an I-cover.
Indeed,
let $\eps \colon E \onto G_0$
be an embedding cover.
By induction there are 
$\psi_k \colon E \onto G_k$
such that
$\psi_0 = \eps$
and
$\pi_k \circ \psi_k = \psi_{k-1}$,
for every $k \ge 1$.
Then $(\psi_k)_{k}$ defines
$\psi \colon E \onto H$
such that
$\pi \circ \psi = \eps$.

(d)
If $G$ is finite then
its smallest embedding cover is finite
(\cite[p. 573]{FJ}).
Hence, by (a), it is unique.

(e)
Let $\pi \colon \Ghat \onto G$ be an epimorphism.
Let
$\epshat \colon \Ehat \onto \Ghat$
and
$\eps \colon E \onto G$
be smallest embedding covers.
As $\pi \circ \epshat$ is an embedding cover
and $\eps$ is an I-cover,
there is an epimorphism
$\Ehat \onto E$.
Therefore
$\calE(\Ghat) \supseteq \calE(G)$
and
$\calE_f(\Ghat) \supseteq \calE_f(G)$.

If $\pi$ is an I-cover,
then $\pi \circ \epshat$ is 
a smallest embedding cover of $G$,
by (b).
Thus
$\calE(\Ghat) = \calE(G)$
and
$\calE_f(\Ghat) = \calE_f(G)$.
\end{Remark}

%
%
%

\begin{Lemma}\label{compact is quotient}
Let
$(\eta_i \colon B_i \onto A )_{i \in I}$
be a finite family
with
$A, B_i \in \calE_f(G)$ for every $i \in I$.
If
$\fprod[A]{i \in I} B_i$
is compact,
then
$\fprod[A]{i \in I} B_i \in \calE_f(G)$.
\end{Lemma}

\begin{proof}
Let $E$ be an embedding cover of $G$.
As $A \in \calE_f(G)$,
there is
$\phi \colon E \onto A$.
As $E$ has the embedding property,
for each $i \in I$ there is
$\psi_i \colon E \onto B_i$
such that
$\eta \circ \psi_i = \phi$.
As $\fprod[A]{i \in I} B_i$ is compact,
the $\psi_i$ define an epimorphism
$\psi \colon E \onto \fprod[A]{i \in I} B_i$.
Thus
$\fprod[A]{i \in I} B_i \in \calE_f(G)$.
\end{proof}

Here is our main source of I-covers:

\begin{Definition}\label{basic cover}
Consider a compact
cartesian square \eqref{basic}
with
$B \in \calE_f(G)$.
Then $\eta$ is an I-cover.

Indeed,
let $\eps \colon E \onto G$ be an embedding cover.
By the embedding property of $E$ there is
$\psi_B \colon E \onto B$
such that $\alpha \circ \psi_B = \phi \circ \eps$.
As \eqref{basic} is compact,
there is
$\psi \colon E \onto H$ such that
$\eta \circ \psi = \eps$
(and $\beta \circ \psi = \psi_B$).

By Remark~\ref{EV trivialities}(e),
$H \in \calE(G)$.
Notice that $\Ker \eta$ is finite,
because $\Ker \eta \isom \Ker \alpha$
and $\Ker \alpha \le B$ is finite.

We call an I-cover $\eta$ of this type \textbf{basic}.
If $B$ is even a finite quotient of $G$
and $\eta$ is indecomposable,
we say that $\eta$ is \textbf{superbasic}.
\end{Definition}

\begin{Remark}\label{nonsplit is basic}
Let $\eta \colon H \onto G$ be an indecomposable epimorphism.
Assume that $H \in \calE(G)$ 
and that $\eta$ does not split.
Then $\eta$ is a basic I-cover.

Indeed,
$\Ker \eta$ is finite,
so there is an open $U \normal H$ such that
$U \cap \Ker \eta = 1$.
Let $B = H/U$ and $A = G/\eta(U)$,
and let 
$\beta \colon H \onto B$ and $\phi \colon G \onto A$
be the quotient maps.
Then there is a cartesian square
\eqref{basic}
in which $\alpha$ is the induced map from $\eta$.
As $H \in \calE(G)$, we have $B \in \calE_f(G)$.

If there were a homomorphism
$\gamma \colon G \to B$ such that $\alpha \circ \gamma = \phi$,
then, as \eqref{basic} is cartesian,
there would exist a homomorphism
$\delta \colon G \to H$
such that
$\eta \circ \delta = \id_{G}$
(and $\beta \circ \delta = \gamma$).
Thus $\eta$ would split, a contradiction.
It follows by Proposition~\ref{indecomposable cartesian square} that
\eqref{basic} is compact.
Hence $\eta$ is basic.
\end{Remark}

\begin{Lemma}\label{embedding covers have no basic covers}
Let $G$ be a profinite group.
The following are equivalent:
\begin{itemize}
\item[(a)]
$G$ has the embedding property.
\item[(b)]
$G$ has no I-covers, except isomorphisms.
\item[(c)]
$G$ has no basic I-covers, except isomorphisms.
\item[(d)]
$G$ has no indecomposable basic I-covers.
\item[(e)]
$G$ has no superbasic I-covers.
\end{itemize}
\end{Lemma}

\begin{proof}
(a) $\implies$ (b):
Let $\eta \colon H \onto G$ be an I-cover.
There is $\psi \colon G \onto H$
such that
$\eta \circ \psi = \id_G$.
Therefore $\eta$ is an isomorphism.

(b) $\implies$ (c) $\implies$ (d) $\implies$ (e):
Trivial.

(e) $\implies$ (a):
Let $\alpha \colon B \onto A$
be an epimorphism of finite quotients of $G$.
We have to show that if
$\phi \colon G \onto A$ is an epimorphism,
there is
$\gamma \colon G \onto B$
with $\alpha \circ \gamma = \phi$.
Writing $\alpha$ as the composition
of finitely many indecomposable epimorphisms
we may assume that $\alpha$ is indecomposable.

Let $H = B \times_A G$.
Then the projection $\eta \colon H \onto G$ is indecomposable.
By (e),
$\eta$ is not a basic cover.
So
the cartesian square defined by $H$ is not compact.
Hence,
by Proposition~\ref{indecomposable cartesian square},
$\alpha \mle \phi$.
\end{proof}

\begin{Construction}\label{sec transfinite}
Let us summarize and refine the construction of
a smallest embedding cover
$\eps \colon E \onto G$ of $G$
from
\cite[Proposition 27.4.5]{FJ}:

There is an ordinal $\mu$
and
a transfinite sequence $\{G_\lambda\}_{\lambda \le \mu}$
of profinite groups 
with 
$G = G_0$
and
$E = G_\mu$
and maps
$\pi_{\lambda, \kappa} \colon G_\lambda \onto G_\kappa$,
for all
$0 \le \kappa \le \lambda \le \mu$,
forming an inverse system,
such that
\begin{itemize}
\item[(a)]
If $\lambda < \mu$,
then $G_\lambda$ has a proper I-cover.

(This condition is not assumed in 
\cite[Proposition 27.4.5]{FJ},
but we can achieve it by replacing the cardinal $m$
given there by the largest ordinal $\mu \le m$ with property (a).)
\item[(b)]
If $\lambda \le \mu$ is a non-limit ordinal, then
$\pi_{\lambda, \lambda-1} \colon G_\lambda \onto G_{\lambda-1}$
is
a proper finite I-cover of $G_{\lambda-1}$.
This I-cover is arbitrarily chosen
in
\cite[Proposition 27.4.5]{FJ}.
Hence,
by Lemma~\ref{embedding covers have no basic covers},
we may 
(and shall)
assume that it is superbasic.

\item[(c)]
If $\lambda \le \mu$ is a limit ordinal, then
$G_\lambda = \varprojlim_{\kappa < \lambda} G_\kappa$
with the projections
$\pi_{\lambda, \kappa} \colon G_\lambda \onto G_\kappa$
of the inverse limit.
\end{itemize}
\end{Construction}

\section{Basic covers}\label{section basic}

We explore properties of
basic covers,
defined in Definition~\ref{basic cover}.

\begin{Lemma}\label{change origin}
Let $\eta \colon H \onto G$
be a basic cover
defined by a compact cartesian square \eqref{basic}.
Then for every open $N \normal G$ contained in $\Ker \phi$
there is a commutative diagram \eqref{raise}
with compact cartesian squares,
$\Ker \phi' = N$,
and
$B' \in \calE_f(G)$.
\end{Lemma}

\begin{proof}
Let $A' = G/N$.
Let $\phi' \colon G \onto A'$ be the quotient map
and 
$\phi'' \colon A' \onto A$
the induced epimorphism
such that
$\phi = \phi'' \circ \phi'$.
By Remark~\ref{square trivialities}(c)
there is a commutative diagram
\begin{equation}\label{raise}
\xymatrix@=24pt{
H \arr[r]_{\eta} \arr[d]^{\beta'}
\arr@/_1pc/[dd]_{\beta}
& G \arr[d]_{\phi'} \arr@/^1pc/[dd]^{\phi}
\\
B' \arr[r]^{\alpha'} \arr[d]^{\beta''} & A' \arr[d]_{\phi''}
\\
B \arr[r]^{\alpha} & A
}
\end{equation}
with all squares cartesian.
By Remark~\ref{square trivialities}(b)
they are compact.
As $B$ and $A'$ are finite, so is $B'$.
As \eqref{basic} is compact,
$H \in \calE(G)$,
hence $B'$,
a finite quotient of $H$, 
is in $\calE_f(G)$.
\end{proof}

\begin{Lemma}\label{compose basic}
Let
$\eta' \colon G' \onto G$
and 
$\eta'' \colon G'' \onto G'$.
Then
$\eta' \circ \eta''$ is basic
if and only if
both $\eta'$ and $\eta''$ are basic.
\end{Lemma}

\begin{proof}
Assume that
$\eta' \circ \eta''$ is basic.
So there is a compact cartesian square
(\emph{without} the dotted arrows)
\begin{equation}\label{joint}
\xymatrix@=24pt{
G''
\arr[r]_{\eta''} \arr[d]^{\phi''}
& G' \arr[r]_{\eta'} \dotarr[d]^{\phi'}
& G \arr[d]^{\phi} 
\\
A''
\arr@/_1pc/[rr]_{\alpha}
\dotarr[r]^{\alpha''} & A' \dotarr[r]^{\alpha'} & A \\
}
\end{equation}
with $A'' \in \calE_f(G)$.
By Remark~\ref{square trivialities}
we can complete the diagram with dotted arrows
such that
all the squares are compact cartesian.
Then 
$A' \in \calE_f(G)$.
Thus 
$\eta'$ and $\eta''$ are basic.

Conversely,
assume that
$\eta'$ and $\eta''$ are basic.
Then
$\eta'$ and $\eta''$ are I-covers,
so by Remark~\ref{EV trivialities}(e),
$\calE(G'') = \calE(G') = \calE(G)$,
and
there are two compact cartesian squares
(on the left)
\begin{equation*}
\xymatrix{
G''
\arr[r]_{\eta''} \arr[d]^{\phi''}
& G' \arr[d]^{\phi'}
\\
A''
\arr[r]^{\alpha''} & A' 
}
\qquad
\xymatrix{
G' \arr[r]_{\eta'} \arr[d]^{\psi}
& G \arr[d]^{\phi} 
\\
B \arr[r]^{\alpha'} & A \\
}
\qquad
\qquad
\qquad
\xymatrix@=12pt{
G' \arr[rr]_{\eta'} \arr[d]^{}
\arr@/_1.5pc/[dd]_{\psi}
&& G \arr[d]_{} \arr@/^1.5pc/[dd]^{\phi}
\\
G'/N' \arr[rr]^{} \arr[d]^{\bar\beta} && G/N \arr[d]_{\bar\phi}
\\
B \arr[rr]^{\alpha'} && A
}
\end{equation*}
with $A'', A', B \in \calE_f(G)$.

Let
$N' = \Ker \phi' \cap \Ker \psi$.
This is an open normal subgroup of $G'$
contained in $\Ker \phi'$.
As $\eta'$ maps isomorphically
$\Ker \psi$ onto $\Ker \phi$,
it maps $N'$ isomorphically onto
$N = \eta'(N')$,
an open normal subgroup of $G$
contained in $\Ker \phi$.
Therefore the upper square in the diagram above on the right
is cartesian
(here $\bar\beta$ and $\bar\phi$
are the induced maps that make the diagram commutative).
By Remark~\ref{square trivialities}
it is compact.
Thus
we may replace $\phi$ and $\psi$
by the quotient maps
$G \onto G/N$ and $G' \onto G'/N'$.
By Lemma~\ref{change origin}
we may replace $\phi'$ by $G' \onto G'/N'$.
So we get \eqref{joint}
with right and left compact squares.
By Remark~\ref{square trivialities}
also their concatenation is a compact square.
Thus $\eta' \circ \eta''$ is basic.
\end{proof}

\begin{Proposition}\label{basic subspace}
Let $(\eta_i \colon H_i \onto G)_{i \in I}$
be a family of indecomposable basic covers of $G$
with the same abelian kernel $C$
(as a simple $G$-module).
Let $F_C = \End_G(C)$.
Then:
\begin{itemize}
\item[(a)]
Every $\eta \in \Span_{F_C}(\eta_i \st i \in I) \smallsetminus \{0\}$
is a basic cover of $G$.
\item[(b)]
If $|I| < \infty$ and
 $(\eta_i)_{i \in I}$
is linearly independent over $F_C$,
then $\eta_I \colon \fprod{i \in I} H_i \onto G$
is a basic cover.
\end{itemize}
\end{Proposition}

\begin{proof}
(a)
Every $\eta \in \Span_{F_C}(\eta_i \st i \in I)$
is a linear combination of finitely many $\eta_i$.
Therefore we may assume that $I$ is finite.
Replacing
$(\eta_i)_{i \in I}$
by a maximal linearly independent subfamily
we may assume that
the $\eta_i$ are linearly independent over $F_C$.
Thus we have to show that
a non-trivial linear combination $\eta$ of the $\eta_i$
is a basic cover of $G$.
By
Proposition~\ref{indecomposable quotients of a fiber product pi is id},
$\eta$ is dominated by
$\eta_I \colon \fprod{i \in I} H_i \onto G$.
So, by Lemma~\ref{compose basic},
it suffices to show (b).

(b)
By Lemma~\ref{change origin}
we may assume that
there are compact cartesian squares \eqref{left}
with $B_i \in \calE_f(G)$,
for every $i \in I$,
with the same $A$ and $\phi$.
Then $\eta_i = \Inf_\phi(\alpha_i)$,
for every $i \in I$.
As
$\Inf_\phi \colon H^2(A,C) \to H^2(G,C)$
is an $F_C$-linear map
and the $\eta_i$ are linearly independent,
so are the $\alpha_i$.

\begin{subequations}
\begin{tabularx}{0.85\textwidth}{Xp{.1cm}Xp{.1cm}Xp{.1cm}X}
\begin{equation}\label{single}
\xymatrix@=34pt{
H \arr[r]_{\eta} \arr[d]^{\beta} & G \arr[d]^{\phi}
\\
B \arr[r]^{\alpha} & A \\
}
\end{equation}
& &
\begin{equation}\label{left}
\xymatrix@=34pt{
H_i \arr[r]_{\eta_i} \arr[d]^{\beta_i} & G \arr[d]^{\phi}
\\
B_i \arr[r]^{\alpha_i} & A \\
}
\end{equation}
& &
\begin{equation}\label{center}
\xymatrix@=20pt{
\fprod{i \in I} H_i
\arr[d]^{\prod \beta_i}
\arr[r]^{\eta_I}
& G \arr[d]^{\phi} 
\\
\fprod[A]{i \in I} B_i
\arr[r]_{\alpha_I}
& A \\
}
\end{equation}
& &
\vspace{-.8cm}
\begin{equation}\label{right}
\xymatrix@=18pt{
\fprod{i \in I} H_i
\arr[r]_{\zeta} \arr[d]^{\prod \beta_i}
\arr@/^1pc/[rr]^{\eta_I}
& H' \arr[r]_{\eta'} \arr[d]_{\beta'}
& G \arr[d]^{\phi} 
\\
\fprod[A]{i \in I} B_i
\arr@/_1pc/[rr]_{\alpha_I}
\arr[r]^{\gamma} & B' \arr[r]^{\alpha'} & A \\
}
\end{equation}
\end{tabularx}
\end{subequations}

By \cite[Corollary~6.11]{fund},
$\fprod{i \in I} H_i$ 
and
$\fprod[A]{i \in I} B_i$
are compact fiber products.
By \cite[Lemma~3.9]{fund},
\eqref{center} is
a compact cartesian square.
By Lemma~\ref{compact is quotient},
$\fprod[A]{i \in I} B_i \in \calE_f(G)$.
Thus $\eta_I$ is basic.
\end{proof}

%
%
%

\begin{Proposition}\label{fprod of basics is basic}
Let
$(\eta_i \colon H_i \onto G)_{i \in I}$ be
a finite family of indecomposable epimorphisms.
Then
$\eta_I \colon \fprod{i \in I} H_i \onto G$
is basic
if and only if
$\fprod{i \in I} H_i \in \calE(G)$
and
every indecomposable cover of $G$
dominated by $\eta_I$
is basic.
\end{Proposition}

\begin{proof}
Let $\calD$ be the set of
(representatives of isomorphism classes of)
indecomposable covers of $G$
dominated by $\eta_I$.
By 
\cite[Lemma 5.1]{fund},
elements of $\calD$ are determined by (certain) epimorphisms
from $\Ker \eta_I$.
As $\Ker \eta_I$ is finite,
$\calD$ is finite.

If $\eta_I$ is basic,
then
$\fprod{i \in I} H_i \in \calE(G)$;
by Lemma~\ref{compose basic},
every $\eta \in \calD$
is basic.

Conversely,
assume that 
$\fprod{i \in I} H_i \in \calE(G)$
and every $\eta \in \calD$ is basic.
By Lemma~\ref{change origin}
we may assume that
there is $\phi \colon G \onto A$
and for every $\eta \in \calD$
there is a compact cartesian square \eqref{single}
with $B$ finite.
In particular, for every $i \in I$
there is a compact cartesian square \eqref{left}
with $B_i$ finite.

If $\phi$ is the composition of some
$\phi' \colon G \onto A'$
and
$\phi'' \colon A' \onto A$
with $A'$ finite,
we may form,
as in the proof of Lemma~\ref{change origin},
a commutative diagram \eqref{raise}
with all squares cartesian
and replace \eqref{single}
with the upper square in \eqref{raise}.
This will serve us to achieve additional properties of
\eqref{single}:

Fix $C \in \Lambda_\ab(G)$,
and let
$I_C = \{i \in I \st \Ker \eta_i \isom_G C\}$
and
$F = \End_G(C)$.
Put
$S_G(C) = \Span_F((\eta_i)_{i \in I_C})
\subseteq
H^2(G,C)$
and
$S_A(C) = \Span_F((\alpha_i)_{i \in I_C})
\subseteq
H^2(A,C)$.
If $\eta \in \calD$
and $\Ker \eta \isom_G C$,
then,
by
Proposition~\ref
{indecomposable quotients of a fiber product pi is id}(b'),(b),
$\eta$ is dominated by $\eta_{I_C} \colon \fprod{i \in I_C} H_i \onto G$
and
$\eta \in S_G(C)$.
The inflation map
$\Inf_\phi \colon H^2(G,A) \to H^2(G,C)$
maps
$S_A(C)$
onto $S_G(C)$
and $\alpha$ onto a multiple of $\eta$,
hence into $S_G(C)$.
As
$H^2(G,C) = \varinjlim_{A'} H^2(A',C)$,
where $A'$ are as in the preceding paragraph,
and $\dim S_{A'}(C) \le |I_C| < \infty$,
we may assume, by the preceding paragraph,
that
\begin{itemize}
\item
[(*)]
$\Inf_\phi$ maps $S_A(C)$
injectively onto $S_G(C)$
and
$\alpha \in S_A(C)$.
\end{itemize}

As $I$ is finite,
$I_C \ne \emptyset$
for only finitely many $C \in \Lambda_\ab(G)$.
Thus we may assume that (*) holds for all $C \in \Lambda_\ab(G)$.

We claim that then the cartesian square \eqref{center} is compact.

If not,
by \cite[Corollary~2.13]{fund}
there is a commutative diagram \eqref{right}
with cartesian squares,
$\alpha', \eta'$ indecomposable,
and the right-handed square not compact.
Let $C = \Ker \alpha'$.

If $C$ is non-abelian,
by
Proposition~\ref{indecomposable quotients of a fiber product pi is id}(a'),
we may assume that $\alpha' = \alpha_i$ for some $i \in I$
and $\gamma = \pr_{I,i}$.
Therefore the right-handed square in \eqref{right} is \eqref{left},
which is compact, a contradiction.

If $C$ is abelian,
by
Proposition~\ref{indecomposable quotients of a fiber product pi is id}(b),
$\alpha' \in S_A(C)$.
By \cite[Lemma~5.2]{fund},
$\eta \isom_G \Inf_\phi(\alpha')$.
Hence, by (*),
$\alpha' \isom_G \alpha$,
whence 
the right-handed square in \eqref{right}
is \eqref{single},
which is compact, a contradiction.

Thus \eqref{center} is compact.
As $\fprod[A]{i \in I} B_i$ is a quotient of 
$\fprod{i \in I} H_i \in \calE(G)$,
it is in $\calE_f(G)$.
Hence $\eta_I$ is basic.
\end{proof}


\begin{Corollary}
Let
$(\eta_i \colon H_i \onto G)_{i \in I}$ be
a finite family of indecomposable epimorphisms.
Then
$\eta_I \colon \fprod{i \in I} H_i \onto G$
is basic
if and only if
$\fprod{i \in I} H_i \in \calE(G)$,
all $\eta_i$ are basic,
and
for every $C \in \Lambda_\ab(G)$
either 
$\mult_C(\eta_I) = 0$
or 
the split extension $\sigma^{(0)}_C \colon G \ltimes C \onto G$
is basic.
\end{Corollary}

\begin{proof}
By Proposition~\ref{fprod of basics is basic}
it suffices to show
that the following are equivalent:
\begin{itemize}
\item[(a)]
Every indecomposable cover $\eta$ of $G$
dominated by $\eta_I$
is basic.
\item[(b)]
All $\eta_i$ are basic,
and
for every $C \in \Lambda_\ab(G)$
either
$\mult_C(\eta_I) = 0$
or 
$\sigma^{(0)}_C$
is basic.
\end{itemize}
For every $C \in \Lambda_\ab(G)$
let
$I_C = \{i \in I \st \Ker \eta_i \isom_G C\}$.

\noindent
(a) $\implies$ (b): 
Every $\eta_i$ is
an indecomposable cover of $G$ dominated by $\eta_I$,
so it is basic.
Let $C \in \Lambda_\ab(G)$
such that
$\mult_C(\eta_I) \ge 1$.
By Lemma~\ref{fiber product is fundamental},
$\eta_I$ is fundamental,
hence by Theorem~\ref{fundamental mult supp}
we may assume that one of the $\eta_i$ is 
$\sigma^{(0)}_C$,
so, again, it is dominated by $\eta_I$, whence basic.

\noindent
(b) $\implies$ (a): 
By
Proposition~\ref{indecomposable quotients of a fiber product pi is id}
every indecomposable cover $\eta$ of $G$ is 
either $\eta_i$,
for some $i \in I$,
or
a non-trivial linear combination of
$(\eta_i)_{i \in I_C}$,
for some $C \in \Lambda_\ab(G)$.
In the latter case,
if $\eta \not\isom_G \sigma^{(0)}_C$,
then $\eta$ is basic
by Proposition~\ref{basic subspace}(a);
if $\eta \isom_G \sigma^{(0)}_C$,
then $(\eta_i)_{i \in I_C}$ is linearly dependent,
hence
by Proposition~\ref{indecomposable quotients of a fiber product pi is id}(b),
$\sigma^{(0)}_C \mle \eta_I$,
whence
$\mult_C(\eta_I) \ge 1$
by Theorem~\ref{fundamental mult supp},
so
$\sigma^{(0)}_C$
is basic by (b).
\end{proof}

\begin{Proposition}\label{going down}
Consider a semi-cartesian square
\begin{equation}\label{zeta}
\xymatrix{
H \arr[r]_{\eta} \arr[d]^{\zeta} & G \arr[d]^{\xi}
\\
\Hbar \arr[r]^{\etabar} & \Gbar \rlap{.}
}
\end{equation}
Assume that
$\calE_f(\Gbar) = \calE_f(G)$.
If $\eta$ is basic,
then $\etabar$ is basic.
\end{Proposition}

\begin{proof}
We may replace $H$ by $H_0 := H/(\Ker \zeta) \cap (\Ker \eta)$,
and $\zeta, \eta$ by the induced maps
$H_0 \onto \Hbar$,
$H_0 \onto G$,
because $H_0 \onto G$ is basic by Lemma~\ref{compose basic}.
Thus we may assume that \eqref{zeta} is a cartesian square.

Let $L = \Ker \xi$
and $M = \Ker \zeta$.
Then
$\eta$ maps $M$ isomorphically onto $L$.
By assumption there is a compact cartesian square \eqref{single}
---the upper face in the diagram \eqref{cube diagram 2} below---
with $\alpha$ indecomposable
and $B \in \calE_f(G) = \calE_f(\Gbar)$.

\subdemoinfo{Case I}{Assume that $L$ is finite.}

There is an open $N \normal G$
such that
$L \cap N = 1$.
Then
$\eta^{-1}(L) \cap \eta^{-1}(N) = \Ker \eta$,
so
$M \cap \eta^{-1}(N) \le \Ker \eta$.
But
$M \cap \Ker \eta = 1$,
so
$M \cap \eta^{-1}(N) = 1$.

By Lemma~\ref{change origin}
we may assume that
$\Ker \phi \le N$.
Then
$L \cap \Ker \phi = 1$
and
$M \cap \beta^{-1}(\Ker \alpha) = M \cap \eta^{-1}(\Ker \phi) = 1$.
So
$\beta(M) \cap \Ker \alpha = 1$.
Furthermore,
$\Ker \beta \le \beta^{-1}(\Ker \alpha)$,
so also
$M \cap \Ker \beta = 1$.
%
%
So the restrictions
$\phi|_L \colon L \onto \phi(L)$,
$\beta|_M \colon M \onto \beta(M)$,
$\alpha|_{\beta(M)} \colon \beta(M) \onto \phi(L)$
are isomorphisms.

Put 
$A_1 = A/\phi(L)$ and $B_1 = B/\beta(M)$,
and let
$\sigma \colon A \onto A_1$
and
$\tau\colon B \onto B_1$
be the quotient maps.
Let
$\phi_1 \colon \Gbar \onto A_1$,
$\beta_1 \colon \Hbar \onto B_1$,
and
$\alpha_1 \colon B_1 \onto A_1$
be the epimorphisms induced from
$\phi$, $\beta$, and $\alpha$,
respectively.
\begin{equation}\label{cube diagram 2}
\xymatrix@=12pt{ 
& H \arr[rr]^(.7){\eta} \arr'[d][dd]^(.4){\zeta}
\arr[ld]_{\beta}
&& G \arr[dd]^(.65){\xi} \arr[ld]^{\phi} \\
B \arr[rr]^(.7){\alpha} \arr[dd]_(.65){\tau}
&& A \arr[dd]_(.65){\sigma} & \\
& \Hbar \arr'[r][rr]^{\etabar} \arr[ld]_(.4){\beta_1}
&& \Gbar \arr[ld]^(.4){\phi_1} \\
B_1 \arr[rr]^{\alpha_1} && A_1 &
}
\end{equation}
Then all the faces in the next cube diagram commute,
except, perhaps, the bottom face.
However, the bottom face commutes as well,
because so do the rest.
So the diagram is commutative.

The back and the upper face are cartesian squares
by assumption.
So are:
the right face
(because 
$\phi$ maps $\Ker \xi = L$ isomorphically onto
$\Ker \sigma = \phi(L)$),
the left face
(because 
$\beta$ maps $\Ker \zeta = M$ isomorphically onto
$\Ker \tau = \beta(M)$),
and
the front face
(because 
$\alpha$ maps $\Ker \tau = \beta(M)$ isomorphically onto
$\Ker \sigma = \phi(L)$).
By Remark~\ref{square trivialities}(a),
the concatenation of the upper and the front face 
is a cartesian square;
this square is also the concatenation of
the back and the bottom face,
hence, again by
Remark~\ref{square trivialities}(a),
the bottom face is a cartesian square.

As the upper face is compact,
by \cite[Lemma~2.14]{fund}
so is the bottom face.
Moreover,
$B_1$ is a quotient of $B$,
hence $B_1 \in \calE_f(G) = \calE_f(\Gbar)$.
Therefore
$\etabar \colon \Hbar \onto \Gbar$ is basic.
This finishes the case of $L$ finite.

\subdemoinfo{Case II}{The general case.}

Let $U = \Ker \beta \cap M$.
Then $U \normal H$ and $U \le \Ker \beta$,
and $U$ is open in $M$.
Also,
$V := \eta(U) \le \Ker \phi$,
because
$\phi(\eta(U)) = \alpha(\beta(U)) = \alpha(1) = 1$,
and $V$ is open in $L$.
Put $H' = H/U$ and $G' = G/V$,
let
$\zeta' \colon H \onto H'$
and
$\xi' \colon G \onto G'$
be the quotient maps.
Let
$\eta', \zeta'', \xi'', 
\beta', \phi'$
be the induced epimorphisms
such that the following two diagrams commute:
\newline
\begin{subequations}
\begin{tabularx}{\textwidth}{Xp{.1cm}X}
\begin{equation}\label{zeta'}
\xymatrix@=24pt{
H \arr[r]_{\eta} \arr[d]^{\zeta'}
\arr@/_1pc/[dd]_{\zeta}
& G \arr[d]_{\xi'} \arr@/^1pc/[dd]^{\xi}
\\
H' \arr[r]^{\eta'} \arr[d]^{\zeta''} & G' \arr[d]_{\xi''}
\\
\Hbar \arr[r]^{\etabar} & \Gbar
}
\end{equation}
& &
\begin{equation}\label{phi'}
\xymatrix@=24pt{
H \arr[r]_{\eta} \arr[d]^{\zeta'}
\arr@/_1pc/[dd]_{\beta}
& G \arr[d]_{\xi'} \arr@/^1pc/[dd]^{\phi}
\\
H' \arr[r]^{\eta'} \arr[d]^{\beta'} & G' \arr[d]_{\phi'}
\\
B \arr[r]^{\alpha} & A
}
\end{equation}
\end{tabularx}
\end{subequations}
The upper square in these diagrams is cartesian,
because
$\Ker \zeta' \cap \Ker \eta \le M \cap \Ker \eta = 1$
and
$\Ker(\xi' \circ \eta) = \eta^{-1}(V) =
U \Ker \eta = (\Ker \zeta')(\Ker \eta)$.
So, by Remark~\ref{square trivialities}(a),
all the squares in these diagrams are cartesian;
by \cite[Remark~2.10(b)]{fund}
the squares in \eqref{phi'} are compact.
From \eqref{zeta'} and from
$\calE_f(\Gbar) = \calE_f(G)$
we deduce that
$\calE_f(\Gbar) = \calE_f(G') = \calE_f(G)$
and from \eqref{phi'} that $\eta'$ is basic.
Observe that
$\Ker \xi'' = L/V$
is finite.
Thus, replacing \eqref{zeta}
with the lower square in \eqref{zeta'},
by Case I we get that
$\etabar$ is basic.
\end{proof}

\begin{Proposition}\label{I is basic}
An indecomposable I-cover is basic.
\end{Proposition}

\begin{proof}
Let 
$\eta \colon H \onto G$
be an indecomposable I-cover.
Assume that it is not basic.

Let $\eps \colon E \onto G$
be the smallest embedding cover of $G$
of Construction~\ref{sec transfinite}.
Thus 
$\mu$ is an ordinal,
there is an inverse system
$(G_\lambda, \pi_{\lambda,\kappa} \colon G_\lambda \onto G_\kappa)_
{0 \le \kappa < \lambda\le \mu}$
such that
$\calE_f(G_\lambda) = \calE_f(G)$
for all $0 \le \lambda \le \mu$,
and
$G = G_0$, $E = G_\mu$.

Put $H_0 = H$ and $\eta_0 = \eta$.
By transfinite induction we now construct,
for every $0 < \lambda \le \mu$,
a non-basic indecomposable
$\eta_\lambda \colon H_\lambda \onto G_\lambda$
such that
\begin{equation}\label{lambda kappa}
\xymatrix{
H_\lambda \arr[d]^{\beta_{\lambda, \kappa}}
\arr[r]^{\eta_\lambda}
& G_\lambda \arr[d]^{\pi_{\lambda, \kappa}}
\\
H_\kappa 
\arr[r]^{\eta_\kappa}
& G_\kappa 
}
\end{equation}
is a compact cartesian square,
for all $0 \le \kappa < \lambda$.

First suppose that
$\lambda$ is not a limit ordinal,
and assume that
a non-basic indecomposable
$\eta_{\lambda-1} \colon H_{\lambda-1} \onto G_{\lambda-1}$
has been constructed
such that 
the bottom square in the diagram below
is compact for every $0 \le \kappa < \lambda-1$.
Let 
$H_\lambda = H_{\lambda-1} \times_{G_{\lambda-1}} G_{\lambda}$
and let
$\beta_{\lambda,\lambda-1}$ and $\eta_{\lambda}$
be the coordinate projections.
Then 
also the upper square in the diagram below
is cartesian.
\begin{equation*}
\xymatrix{
H_\lambda \arr[d]^{\beta_{\lambda, \lambda-1}}
\arr[r]^{\eta_\lambda}
& G_\lambda \arr[d]^{\pi_{\lambda, \lambda-1}}
\\
H_{\lambda-1} \arr[d]^{\beta_{\lambda-1, \kappa}}
\arr[r]^{\eta_{\lambda-1}}
& G_{\lambda-1} \arr[d]^{\pi_{\lambda-1, \kappa}}
\\
H_\kappa 
\arr[r]^{\eta_\kappa}
& G_\kappa 
}
\end{equation*}
There cannot be
$\gamma \colon G_\lambda \onto H_{\lambda-1}$
such that
$\eta_{\lambda-1} \circ \gamma = \pi_{\lambda, \lambda-1}$,
because, as 
$\pi_{\lambda, \lambda-1}$ is indecomposable
and $\eta_{\lambda-1}$ is not an isomorphism,
$\gamma$ would be an isomorphism,
but 
$\pi_{\lambda, \lambda-1}$ is basic, while
$\eta_{\lambda-1}$ is not, a contradiction.
Thus, by Proposition~\ref{indecomposable cartesian square},
the upper square 
is compact.
Finally, by Remark~\ref{square trivialities}(b),
\eqref{lambda kappa}
is compact.

Now
suppose that $\lambda$ is a limit ordinal.
Let
$\eta_{\lambda} \colon H_{\lambda} \onto G_{\lambda}$
be the inverse limit of
$\eta_{\kappa} \colon H_{\kappa} \onto G_{\kappa}$
for all $\kappa < \lambda$.
Then
\eqref{lambda kappa}
is a compact cartesian square,
for all $\kappa < \lambda$,
by \cite[Lemma~2.15]{fund}.

In both cases,
by Proposition~\ref{going down},
$\eta_\lambda$ is not basic.
This finishes the construction.

Consider
\eqref{lambda kappa}
with $\lambda = \mu$ and $\kappa = 0$.
Then the square below is compact.
\begin{equation*}
\xymatrix{
E \ar@/^1pc/@{=}[drr]
\dotarr@/_1pc/[ddr]_{\gamma_0}
\dotarr[dr]^{\gamma}
\\
& H_\mu \arr[d]^{\beta_{\mu, 0}}
\arr[r]^{\eta_\mu}
& G_\mu \arr[d]^{\pi_{\mu, 0}}
\\
& H
\arr[r]^{\eta}
& G 
}
\end{equation*}
As $G_\mu = E$ and $\pi_{\mu, 0}$ is a smallest embedding cover,
and $\eta$ is an I-cover,
there is $\gamma_0 \colon E \onto H$
such that
$\eta \circ \gamma_0 = \pi_{\mu, 0}$.
As the square is compact,
there is $\gamma \colon E \onto H_\mu$
such that
$\eta_\mu \circ \gamma$ is the identity of $E = G_\mu$
and $\beta_{\mu, 0} \circ \gamma = \gamma_0$.
It follows that the indecomposable $\eta_\mu$ is an isomorphism,
a contradiction.
\end{proof}

The following result is a corollary of
Lemma~\ref{quotient of fprod not I}.

\begin{Corollary}\label{not basic}
Consider a cartesian square
\begin{equation}\label{ss}
\xymatrix{
\Hhat \arr[r]_{\etahat} \arr[d]^{\beta} & \Ghat \arr[d]^{\pi}
\\
H \arr[r]^{\eta} & G
}
\end{equation}
with indecomposable $\eta, \etahat$.
Suppose that $\pi$ is an I-cover
that dominates
a fiber product
$\pi' \colon H' \onto G$
of infinitely many copies of $\eta$.
Then $\etahat$ is not basic.
\end{Corollary}

\begin{proof}
By assumption there is
$\pihat \colon \Ghat \onto H'$
such that 
$\pi' \circ \pihat = \pi$.
By Remark~\ref{square trivialities}(c)
there is a commutative diagram
\begin{equation*}
\xymatrix@=24pt{
\Hhat \arr[r]_{\etahat} \arr[d]^{\hat\beta}
\arr@/_1pc/[dd]_{\beta}
& \Ghat \arr[d]_{\pihat} \arr@/^1pc/[dd]^{\pi}
\\
\Hhat' \arr[r]^{\eta'} \arr[d]^{\beta'}
& H'
 \arr[d]_{\pi'}
\\
H \arr[r]^{\eta} & G
}
\end{equation*}
with all squares cartesian.
By Remark~\ref{EV trivialities}(e),
$\calE_f(G) \subseteq \calE_f(H') \subseteq \calE_f(\Ghat)$,
but 
$\calE_f(G) =  \calE_f(\Ghat)$,
because $\pi$ is an I-cover,
so
$\calE_f(H') = \calE_f(\Ghat)$.
By Proposition~\ref{going down}
it suffices to show that $\eta'$ is not basic.

So
$\pi'$
is the fiber product
$H' = \fprod{i \in I'} H_i \onto G$
of an infinite family
$(\eta_i \colon H_i \onto G)_{i \in I'}$
such that
$\eta_i \isom_G \eta$
for every $i \in I'$.
Let 
$\eta_\ell \colon H_\ell \onto G$
be another copy of $\eta$
and let
$I = I' \dotcup \{\ell\}$.
By \cite[Remark~3.2(e)]{fund}
we may assume that
$\Hhat' = \fprod{i \in I} H_i$
and $\eta' = \pr_{I,I'}$.
By Lemma~\ref{quotient of fprod not I},
$\eta'$ is not basic.
\end{proof}

\section{Uniqueness of a smallest embedding cover}
\label{section uniqueness}

\begin{Lemma}\label{fec is unique}
Let
$\pi \colon E \onto G$
and
$\pi' \colon E' \onto G$
be two smallest embedding covers.
Then their respective fundaments
$\bar\pi \colon \Ebar \onto G$
and
$\bar\pi' \colon \Ebar' \onto G$
are isomorphic,
that is,
there is an isomorphism
$\bar\theta \colon \Ebar \to \Ebar'$
such that
$\bar\pi' \circ \bar\theta = \bar\pi$.
\end{Lemma}

\begin{proof}
By the definition of a
smallest embedding cover
there is 
$\rho \colon E \onto E'$
such that
$\pi' \circ \rho = \pi$.
By \cite[Theorem~9.4(a)]{fund}
there is
$\rhobar \colon \Ebar \onto \Ebar'$
such that
$\bar\pi' \circ \rhobar = \bar\pi$.
By a symmetrical argument
there is 
$\rhobar' \colon \Ebar' \onto \Ebar$
such that
$\bar\pi \circ \rhobar' = \bar\pi'$.
By \cite[Corollary~7.3]{fund}
there is an isomorphism
$\bar\theta \colon \Ebar \to \Ebar'$
such that
$\bar\pi' \circ \bar\theta = \bar\pi$.
\end{proof}

\begin{Proposition}\label{special sec}
Every profinite group $G$ has a 
smallest embedding cover
$\pi \colon H \onto G$,
such that
its $k$-th fundament
$\pi_k \colon G_k \onto G_{k-1}$
is an I-cover,
and its $\lambda$-multiplicity is at most $\aleph_0$,
for every $\lambda \in \Lambda(G_{k-1})$
and each $k \ge 1$.
\end{Proposition}

\begin{proof}
The proof consists of three parts:

\subdemoinfo{Part A}
{Construction.}
Let $G_0 = G$.
Suppose, by induction on $k \ge 0$,
that we have constructed
$G_{k-1}$
and epimorphisms
$\pi_i \colon G_i \onto G_{i-1}$
for $i < k$.
Let
$\pihat_{k} \colon \Ghat_{k} \to G_{k-1}$ be the fundament of 
a smallest embedding cover of $G_{k-1}$
(by Lemma~\ref{fec is unique}
this smallest embedding cover
need not be specified).
By Remark~\ref{EV trivialities}(b),
$\pihat_{k}$ is an I-cover.

We define $\pi_{k} \colon G_{k} \to G_{k-1}$ by its 
multiplicities and supports:
$\mult_\lambda(\pi_{k}) = 
\min(\mult_\lambda(\pihat_{k}), \aleph_0)$
for every
$\lambda \in \Lambda(G_{k-1})$,
and
$\supp_C(\pi_{k}) = \supp_C(\pihat_{k})$
for every
$C \in \Lambda_\ab(G_{k-1})$.
By Theorem~\ref{compare}(1),
$\pi_{k} \mle \pihat_{k}$
and
every fundamental
$\zeta \colon H' \onto G_{k-1}$
with finite kernel
(and, hence, with finite multiplicities)
is dominated by $\pi_{k}$
if and only if it
is dominated by $\pihat_{k}$.
By Remark~\ref{EV trivialities}(b),
$\pi_{k}$ is an I-cover.

Let $\pi \colon H \onto G_0$ be the inverse limit of 
\begin{equation}\label{pi sequence}
\xymatrix{
\cdots \arr[r]^{\pi_3}
& G_2 \arr[r]^{\pi_2}
& G_1 \arr[r]^{\pi_1}
& G_0
}
\end{equation}
and let
$\rho_k \colon H \onto G_k$
be its projections,
for every $k \ge 0$.
By Remark~\ref{EV trivialities}(c),
$\pi$ is an I-cover.
By Remark~\ref{EV trivialities}(e),
$\calE_f(H) = \calE_f(G_k)$,
for every $k \ge 0$.

\subdemoinfo{Part B}
{$H$ has the embedding property.}
Indeed,
let $(\phi \colon H \onto A, \alpha \colon B \onto A)$
be a finite embedding problem,
with $\alpha$ indecomposable
and $B \in \calE_f(H)$.
As $A$ is finite,
there is $k$ and an epimorphism
$\phi_k \colon G_k \onto A$
such that
$\phi = \phi_k \circ \rho_k$.
If $(\phi_k, \alpha)$ has a solution $\gamma_k$,
then $\gamma_k \circ \rho_k$ solves $(\phi, \alpha)$.

If not,
the cartesian square
defined by $\phi_k$ and $\alpha$
(in the diagram below on the left,
with $\beta$ and $\zeta$)
is compact,
by Proposition~\ref{indecomposable cartesian square},
and hence $\zeta$ is basic.
So $\zeta$ is an indecomposable I-cover,
and therefore
$\zeta \mle \pihat_{k+1}$.
As noted above, this implies that
$\zeta \mle \pi_{k+1}$,
that is, there is
$\theta' \colon G_{k+1} \onto H'$
such that
$\zeta \circ \theta' = \pi_{k+1}$.
\begin{equation*}
\xymatrix{
& G_{k+1} \dotarr[dl]_{\theta'} \arr[d]^{\pi_{k+1}}
& & H  \arr@/^2pc/[ddll]^{\phi}
 \arr@/^1pc/[dll]^{\rho_k}
 \arr[ll]_{\rho_{k+1}}
\\
H' \arr[r]^{\zeta} \arr[d]^{\beta} & G_k \arr[d]_{\phi_k} \\
B \arr[r]^{\alpha} & A
}
\quad
\quad
\quad
\quad
\xymatrix@=12pt{
G_{k+1} \dotarr[dr]^{\gamma} \arr@/^1pc/[drrr]^{\pi_{k+1}}
\arr@/_1pc/[dddr]_{\tau}
\\
& H_k \arr[rr]^{\eta_{k}} \arr[dd]_{\beta_k} \dotarr[rrdd]^{\xi_k}
&& G_k \arr[dd]^{\pi_k}
\\
\\
& H_{k-1} \arr[rr]_{\eta_{k-1}} && G_{k-1}
}
\end{equation*}
Then $\beta \circ \theta' \circ \rho_{k+1}$
is a solution to 
$(\phi, \alpha)$.

\subdemoinfo{Part C}
{\eqref{pi sequence}
is the fundament sequence of $\pi \colon H \onto G_0$.}
By \cite[Proposition~9.5]{fund}
it suffices to show that
for no $k \ge 1$ there is a semi-cartesian square
\begin{equation}\label{small}
\xymatrix{
G_{k+1} \arr[r]^{\pi_{k+1}} \arr[d]_{\tau} & G_k \arr[d]^{\pi_k}
\\
H_{k-1} \arr[r]_{\eta_{k-1}} & G_{k-1}
}
\end{equation}
with indecomposable $\eta_{k-1}$.

Suppose there is such a diagram.
Let $H_k = H_{k-1} \times_{G_{k-1}} G_k$
with projections 
$\beta_k$ and $\eta_k$
(the diagram above on the right)
and put
$\xi_k = \pi_k \circ \eta_k  = \eta_{k-1} \circ \beta_k
\colon H_k \onto G_{k-1}$.

By Lemma~\ref{expand},
the unique 
$\gamma \colon G_{k+1} \to H_k$
such that the diagram commutes
is surjective.
As $\pi_k, \pi_{k+1}$ are I-covers,
by Remark~\ref{EV trivialities}(b),(c)
so are
$\eta_k$,
$\pi_k \circ \pi_{k+1}$,
and hence so are
$\eta_{k-1}$ and $\xi_k$.
By Proposition~\ref{I is basic},
$\eta_k$ and $\eta_{k-1}$ are basic.

For a cardinality $\kappa$ let
$\eta_{k-1}^{(\kappa)}$
denote
the fiber product of $\kappa$ copies of $\eta_{k-1}$.
By Corollary~\ref{dominates multiple of one},
the set
$\{\kappa \st \eta_{k-1}^{(\kappa)} \mle \pi_k\}$
has a maximum;
let $\kappa$ be this maximum.
If $\kappa$ is infinite,
then,
by Corollary~\ref{not basic},
$\eta_{k}$ is not basic, a contradiction.
Suppose $\kappa < \aleph_0$.
By Lemma~\ref{expand},
$\xi_k \colon H_k \onto G_{k-1}$
dominates
the fiber product
of $\eta_{k-1}$ and 
$\eta_{k-1}^{(\kappa)}$,
which,
by \cite[Remark 3.2]{fund},
is isomorphic to
$\eta_{k-1}^{(\kappa+1)}$.
As a fundamental I-cover,
$\xi_k$ is dominated by $\pihat_k$,
so,
$\eta_{k-1}^{(\kappa+1)} \mle \pihat_k$.
But,
as remarked in Part A,
this implies
$\eta_{k-1}^{(\kappa+1)} \mle \pi_k$,
a contradiction to the definition of $\kappa$.
\end{proof}

\begin{Lemma}\label{I dominates basic}
A fundamental I-cover with finite kernel is basic.
\end{Lemma}

\begin{proof}
Let $\tau \colon H \onto G$
be a fundamental I-cover with finite kernel.
Let $\eta$ be an indecomposable epimorphism,
$\eta \mle \tau$.
By Remark~\ref{EV trivialities}(b),
$\eta$ is an I-cover.
In particular,
$H \in \calE(G)$.
By Proposition~\ref{I is basic}, $\eta$ is basic.
By Theorem~\ref{fundamental mult supp},
$\tau$ is a fiber product of
indecomposable epimorphisms dominated by $\tau$.
Thus, by Proposition~\ref{fprod of basics is basic},
$\tau$ is basic.
\end{proof}

\begin{Theorem}\label{aleph}
Let \eqref{pi sequence}
be the fundament series of a smallest embedding cover
$\pi \colon H \onto G$.
Let $k \ge 1$.
Then
\begin{itemize}
\item[(a)]
$\mult_\lambda(\pi_k) \le \aleph_0$
for every $\lambda \in \Lambda(G_{k-1})$;
\item[(b)]
every cover of $G_{k-1}$ with finite kernel
dominated by $\pi_k$ is basic.
\end{itemize}
\end{Theorem}

\begin{proof}
We switch the notation:
Let $G' = G$.
Let
$\pi' \colon H' \onto G'$ be a
smallest embedding cover of $G'$
and let
$
\xymatrix{
\cdots \arr[r]^{\pi'_3}
& G'_2 \arr[r]^{\pi'_2}
& G'_1 \arr[r]^{\pi'_1}
& G'_0 = G'
}
$
be its fundament series.
We prove (a) and (b) for this cover.

Let \eqref{pi sequence}
be the fundament series of the special smallest embedding cover
of Proposition~\ref{special sec}.
$\pi \colon H \onto G$.
There is $\theta \colon H \onto H'$
such that
$\pi' \circ \theta = \pi$.
By \cite[Theorem~9.4(b)]{fund}
there is a commutative diagram
\begin{equation}\label{two series diagram}
\xymatrix{
H \arr@/^1.1pc/[rr]_{\rho_k}
\arr@/^1.7pc/[rrr]^(.7){\rho_{k-1}}
\arr[d]_{\theta}
& \cdots \arr[r]_{\pi_{k+1}}
& G_{k} \arr[r]^{\pi_{k}}
\arr[d]_{\theta_{k}}
& G_{k-1} \arr[r]_{\pi_{k-1}}
\arr[d]_{\theta_{k-1}}
& \cdots \arr[r]_{\pi_{3}}
& G_{2} \arr[r]_{\pi_{2}}
\arr[d]_{\theta_{2}}
& G_{1} \arr[r]_{\pi_{1}}
\arr[d]_{\theta_{1}}
& G_{0}
\ar@{=}[d]_{\theta_{0}}
\\
H' \arr@/_1.1pc/[rr]^{\rho'_k}
\arr@/_1.7pc/[rrr]_(.7){\rho'_{k-1}}
& \cdots \arr[r]^{\pi'_{k+1}}
& G'_{k} \arr[r]^{\pi'_{k}}
& G'_{k-1} \arr[r]^{\pi'_{k-1}}
& \cdots \arr[r]^{\pi'_{3}}
& G'_{2} \arr[r]^{\pi'_{2}}
& G'_{1} \arr[r]^{\pi'_{1}}
& G'_{0}
}
\end{equation}
with semi-cartesian squares.

(a)
Let $\lambda' \in \Lambda(G'_{k-1})$.
By \cite[Theorem~7.4]{fund},
$\mult_{\lambda'}(\pi'_k) \le \mult_{\lambda}(\pi_k)$
for some
$\lambda \in \Lambda(G_{k-1})$.
By Proposition~\ref{special sec},
$\mult_{\lambda}(\pi_k) \le \aleph_0$,
hence
$\mult_{\lambda'}(\pi'_k) \le \aleph_0$.

(b)
Let
$\tau' \colon D' \onto G'_{k-1}$
be a cover of $G'_{k-1}$ with finite kernel,
dominated by $\pi'_k$.
Thus there is $\gamma' \colon G'_k \onto D'$
such that
$\pi'_k = \tau' \circ \gamma'$.
\begin{equation*}
\xymatrix@=24pt{
G_k \arr[rr]_{\theta_k} \arr[d]^{\gamma}
\arr@/_1pc/[dd]_{\pi_k}
&& G'_k \arr[d]_{\gamma'} \arr@/^1pc/[dd]^{\pi'_k}
\\
D \arr[rr]^{\thetabar} \arr[d]^{\tau} && D' \arr[d]_{\tau'}
\\
G_{k-1} \arr[rr]^{\theta_{k-1}} && G'_{k-1}
}
\end{equation*}
Let $D = G_{k-1} \times_{G'_{k-1}} D'$,
with projections $\tau$ and $\thetabar$.
Then
$\Ker \tau \isom \Ker \tau'$ is finite
and
there is $\gamma \colon G_k \to D$ that makes the above
diagram commutative.
The square
$(\theta_k, \pi'_k, \pi_k, \theta_{k-1})$
is semi-cartesian,
hence, by Lemma~\ref{expand},
$\gamma$ is surjective.
By Proposition~\ref{special sec},
$\pi_k$ is an I-cover,
hence
by Lemma~\ref{I dominates basic},
$\tau$ is basic.
Obviously
$\calE_f(G_{k-1}) =  \calE_f(G) = \calE_f(G'_{k-1})$,
hence by Proposition~\ref{going down},
$\tau'$ is basic.
\end{proof}

\begin{Corollary}\label{finite quotients of sec-k}
Let
$\pi \colon H \onto G$
be a smallest embedding cover,
and let
\eqref{pi sequence}
be its fundament series.
Let $k \ge 1$
and let $\rho_{k-1} \colon H \onto G_{k-1}$
be the projection.
Let $\tau \colon H' \onto G_{k-1}$
be fundamental, with finite kernel.
The following are equivalent:
\begin{itemize}
\item[(a)]
$\tau$ is basic.
\item[(b)]
$\tau$ is an I-cover.
\item[(c)]
$\rho_{k-1}$ dominates $\tau$.
\item[(d)]
$\pi_k$ dominates $\tau$.
\end{itemize}
\end{Corollary}

\begin{proof}
(a) $\implies$ (b):
Clear.

\noindent
(b) $\implies$ (c):
As $\rho_{k-1}$ is an embedding cover,
it dominates an I-cover.

\noindent
(c) $\implies$ (d):
There is 
$\theta \colon H \onto H'$
such that
$\tau \circ \theta = \rho_{k-1}$.
Now,
$\pi_k$ is the fundament of $\rho_{k-1}$;
As $\tau$ is fundamental,
its fundament is $\tau$.
Thus by \cite[Theorem~9.4(a)]{fund}
there is
$\theta_1 \colon G_1 \onto H'$
such that
$\tau \circ \theta_1 = \pi_k$.

\noindent
(d) $\implies$ (a):
This is Theorem~\ref{aleph}(b).
\end{proof}

For an indecomposable cover $\eta$ of $G$
and an integer $n \ge 0$
let
$\eta^{(n)}$,
denote the fiber product
of $n$ copies of $\eta$.

\begin{Theorem}\label{complete characterization}
Let \eqref{pi sequence}
be the fundament series of a smallest embedding cover
$\pi \colon H \onto G$.
Let $k \ge 1$.
Then:
\begin{itemize}
\item[(a)]
If $\eta \in \Lambda_\na(G_{k-1})$,
then $\mult_\eta(\pi_k) =
\sup(n \ge 0 \st \eta^{(n)} \textnormal{ is basic })$.
\item[(b)]
If $A \in \Lambda_\ab(G_{k-1})$,
then $\mult_A(\pi_k) =
\sup(n \ge 0 \st \eta^{(n)} \textnormal{ is basic })$,
where $\eta$ is the split extension of $G_{k-1}$
with kernel $A$.
\item[(c)]
for every $A \in \Lambda_\ab(G_{k-1})$
and
every $\eta \in H^2(G_{k-1},A) \smallsetminus \{0\}$
we have
$\eta \in \supp_A(\pi_k)$
if and only if
$\eta$ is basic.
\end{itemize}
\end{Theorem}

\begin{proof}
(a)
By Theorem~\ref{aleph},
$\mult_\eta(\pi_k) \le \aleph_0$.
Therefore,
by Corollary~\ref{dominates multiple of one}(a),
$\mult_\eta(\pi_k) = \sup(n \ge 0 \st \eta^{(n)} \mle \pi)$.
So (a) follows by
Corollary~\ref{finite quotients of sec-k}.

(b)
Similarly,
with
Corollary~\ref{dominates multiple of one}(b)
instead of
Corollary~\ref{dominates multiple of one}(a).

(c)
By Corollary~\ref{dominates multiple of one}(b),
$\supp_A(\pi_k) \smallsetminus \{0\}$ consists of
all indecomposable non-split extensions of $G$
with kernel $A$
dominated by $\pi_k$.
Hence (c) follows from Corollary ~\ref{finite quotients of sec-k}.
\end{proof}

\begin{Theorem}\label{uniqueness}
A smallest embedding cover
of a profinite group
is unique,
up to an isomorphism.
\end{Theorem}

\begin{proof}
Let \eqref{pi sequence}
be the fundament series of a smallest embedding cover
$\pi \colon H \onto G$.
By Theorem~\ref{complete characterization},
the multiplicities and supports of each
$\pi_k \colon G_k \onto G_{k-1}$
are uniquely determined
by the basic covers of $G_{k-1}$.
By Theorem~\ref{compare}(ii),
$\pi_k$ is determined,
up to a $G_{k-1}$-isomorphism,
by its multiplicities and supports.
Thus $\pi_k$ is uniquely determined.
Therefore $\pi$ is uniquely determined.
\end{proof}

\section{Examples}\label{examples}

We want to construct examples of I-covers.
We begin with a characterization:

\begin{Lemma}\label{I-cover}
An epimorphism
$\pi' \colon H' \onto G$
is an I-cover
if and only if
\begin{itemize}
\item[(a)]
$\mult_\lambda(\pi') \le \aleph_0$,
for every $\lambda \in \Lambda(G$);
and
\item[(b)]
Let $\eta$ be an indecomposable cover of $G$
and $n \ge 0$ an integer.
If $\eta^{(n)} \mle \pi'$,
then $\eta^{(n)}$ is basic.
\end{itemize}
\end{Lemma}

\begin{proof}
Let $\pi$ be the smallest embedding cover of $G$
and let
$\pi_1$ be its fundament.

If $\pi'$ is an I-cover,
then
$\pi' \mle \pi$.
By Remark~\ref{hereditary},
$\pi' \mle \pi_1$.
Hence (a), (b) follow from Theorem~\ref{aleph}.

Conversely,
assume (a), (b).
By Proposition~\ref{special sec}
(and by the uniqueness of $\pi$),
$\pi_1$ is an I-cover,
hence it suffices to show that
$\pi' \mle \pi_1$,
that is,
by Theorem~\ref{compare}(i),
that
\begin{itemize}
\item[(i1)]
$\mult_\lambda(\pi') \le \mult_\lambda(\pi_1)$,
for every $\lambda \in \Lambda(G)$;
and
\item[(i2)]
$\supp_A(\pi') \subseteq \supp_A(\pi_1)$,
for every $A \in \Lambda_\ab(G)$.
\end{itemize}

We prove (i2):
Let $A \in \Lambda_\ab(G)$
and let
$0 \ne \eta \in \supp_A(\pi')$.
By Corollary~\ref{dominates multiple of one},
$\eta \mle \pi'$,
hence, by (b),
$\eta$ is basic.
In particular,
it is an I-cover,
so $\eta \mle \pi$,
whence
$\eta \mle \pi_1$
by Remark~\ref{hereditary}.
By Corollary~\ref{dominates multiple of one},
$\eta \in \supp_A(\pi_1)$.

We prove (i1):
Let $\lambda \in \Lambda(G)$
and denote 
$n = \mult_\lambda(\pi_1)$.
By (a) we may assume that
$n < \aleph_0$.
If $\lambda = A \in \Lambda_\ab(G)$,
let $\eta$ be the split cover of $G$ with kernel $A$;
if $\lambda \in \Lambda_\na(G)$,
put $\eta = \lambda$.
By Corollary~\ref{dominates multiple of one},
$\eta^{(n+1)} \not\mle \pi_1$,
hence, by Remark~\ref{hereditary},
$\eta^{(n+1)} \not\mle \pi$,
whence $\eta^{(n+1)}$ is not basic.
By (b),
$\eta^{(n+1)}\not\mle \pi'$,
so, 
by Corollary~\ref{dominates multiple of one},
$n+1 \not\le \mult_\lambda(\pi')$,
that is,
$\mult_\lambda(\pi') \le n$.
\end{proof}

We give a general construction of examples of I-covers.

\begin{Lemma}\label{general}
Let
$(\alpha_i \colon B_i \onto A)_{i \in I}$
be a family
of indecomposable epimorphisms of finite groups,
with $|I| \le \aleph_0$.
Assume that
\begin{itemize}
\item[(a)]
no $B_i$ is the direct product of
$\Ker \alpha_i$ with some subgroup of $B_i$;

\item[(b)]
for every simple \textsl{trivial} $A$-module $K$,
the family
$(\alpha_i \st \Ker \alpha_i \isom_A K)$
is linearly independent over $\End_A(K)$.
\end{itemize}
Let 
$G = \fprod[A]{i \in I} B_i \times A$,
and let $\phi \colon G \onto A$
be the projection on the second coordinate.
For every $i \in I$
let $\eta_i \colon H_i \onto G$
be the map defined by the following cartesian square
\begin{equation*}
\xymatrix{
H_i \arr[r]^{\eta_i} \arr[d]^{\beta_i} & G \arr[d]^{\phi}
\\
B_i \arr[r]^{\alpha_i} & A
}
\end{equation*}
Then
$\eta_I \colon \fprod{i \in I} H_i \onto G$
is an I-cover.
\end{Lemma}

\begin{proof}
Let 
$B = \fprod[A]{i \in I} B_i$
and
$H = \fprod{i \in I}H_i$.

\subdemoinfo{Claim A}
{Let
$\phi' \colon G \onto D$
and
$\pi \colon D \onto A$
be 
such that
$\pi \circ \phi' = \phi$.
Put
$K = \phi'(B)$
and
$L := \phi'(A)$.
Then
$D = K \times L$
and
$\Ker \pi = K$.
}

Indeed,
$\Ker \phi' \le \Ker \phi = B$,
so
$(\phi')^{-1}(K) = (\phi')^{-1}(\phi'(B)) = B \Ker \phi' =  B$.
Hence, as 
$B \cap A = 1$,
we have
$K \cap L = 1$.
As $B,A \normal G$ and $G = BA$,
we have
$K,L \normal D$
and
$D= KL$.
Thus
$D = K \times L$.
Finally,
$\Ker \pi = \phi'(\Ker \phi) = \phi'(B) = K$.

\subdemoinfo{Claim B}
{
The cartesian square in the diagram below 
is compact.
\begin{equation}\label{ccc}
\xymatrix@=5pt{
\llap{$H :=$ }
\fprod{i \in I}H_i \arr[rrr]^{\eta_{I}} \arr[ddd]_{\prod \beta_i}&&&
G \arr[ddd]^{\phi} \dotarr[ddl]_{\phi'}
\\
\\
&& D\arr@{.>}[dr]^{\pi}
\\
\llap{$B :=$ }
\fprod[A]{i \in I}B_i \arr[rrr]_{\alpha_{I}} \dotarr[rru]^{\alpha'}
&&& A
}
\end{equation}
}

Indeed,
if it is not compact,
there is an indecomposable epimorphism
$\pi \colon D \onto A$
and epimorphisms
$\alpha' \colon B \onto D$,
$\phi' \colon G \onto D$
such that
$\pi \circ \alpha' = \alpha_{I}$
and
$\pi \circ \phi' = \phi$.
Let $K = \Ker \pi$.
By Claim A,
$D$ is the direct product of
$K$ and a subgroup of $D$.
In particular,
$\pi$ splits.

On the other hand,
as $\pi$ is indecomposable and $\pi \mle \alpha_{I}$,
by Proposition~\ref{indecomposable quotients of a fiber product pi is id}
$\pi$ is isomorphic either to $\alpha_i$,
for some $i \in I$,
which contradicts (a),
or,
if $K$ is abelian,
to a non-trivial linear combination of
$(\alpha_i \st i \in I,\ \Ker \alpha_i \isom_A K)$.
In the latter case
$K$ is a trivial $A$-module
and $\pi$ corresponds to $0$ in $H^2(A,K)$,
by the preceding paragraph,
which contradicts (b).

\subdemoinfo{Claim C}
{
If $I$ is finite, then $\eta_I$ is an I-cover.
}

Indeed,
let $\eps \colon E \onto G$ be an embedding cover.
Then (referring to diagram~\eqref{ccc})
$B$ is a finite quotient of $G$ and hence of $E$,
so there is an epimorphism
$\gamma' \colon E \onto B$ 
such that
$\alpha_{I} \circ \gamma' = \phi \circ \eps$.
As the square in \eqref{ccc} is compact,
the unique homomorphism
$\gamma \colon E \to H$ 
such that
$\prod \beta_i \circ \gamma = \gamma'$
and
$\eta_{I} \circ \gamma = \eps$
is surjective.

\subdemoinfo{Claim D}
{
$\eta_I$ is an I-cover.
}

It suffices to verify conditions (a), (b) of 
Lemma~\ref{I-cover}
for $\pi' = \eta_I$.
The first one holds because $|I| \le \aleph_0$.
Let $\eta$ be an indecomposable cover of $G$
and assume that
$\eta^{(n)} \mle \eta_I$.
As $\Ker \eta^{(n)}$ is finite
and $H = \varprojlim_{J} \fprod{i \in J} H_i$,
where $J$ runs through all finite subsets of $I$,
there is $J$
such that
$\eta^{(n)} \mle \eta_J \colon \fprod{i \in J} H_i \onto G$.
Therefore condition (b) of
Lemma~\ref{I-cover}
holds by Claim C.
\end{proof}

\begin{Example}\label{examples of I-covers}
One could wonder whether in condition (a)
of Theorem~\ref{aleph}
one could write
$\mult_\lambda(\pi_k) < \aleph_0$
instead of
$\mult_\lambda(\pi_k) \le \aleph_0$.
This is not the case:
We use
Lemma~\ref{general}
to produce examples of I-covers $\pi$
with
$\mult_\lambda(\pi) = \aleph_0$ for some $\lambda$.

Let $C_n$ denote the cyclic group of order $n$.

(a)
Let $I = \N$
and for every $i \in I$ let
$B_i = S_3$
and let
$\alpha_i \colon B_i \onto A$
be the epimorphism $S_3 \onto C_2$.
Then $\alpha_i$ is split,
with a non-trivial kernel.
It follows that $\eta_i \colon H_i \onto G$
in Lemma~\ref{general}
is split 
with a non-trivial kernel $K = C_3$,
and it is the same map for every $i$.
Conditions (a), (b) of Lemma~\ref{general}
are satisfied,
so
$\eta_I \colon \fprod{i \in I} H_i \onto G$
is an I-cover with
$\mult_K(\eta_I) = \aleph_0$.

(b)
Similarly,
with $S_n$ instead of $S_3$, for $n \ge 5$,
we have
$\mult_\eta(\eta_I) = \aleph_0$
for the indecomposable epimorphisms
$\alpha_i \colon S_n \onto C_2$.
\end{Example}

\begin{Example}\label{composition}
There is an I-cover
that is a composition $\zeta \circ \xi$
of two epimorphisms
and $\xi$ is not an I-cover.

Indeed,
by Example~\ref{examples of I-covers}
there is
a profinite group $G$
and
an infinite family of isomorphic indecomposable epimorphisms
$(\eta_i \colon H_i \onto G)_{i \in I}$
such that
$\eta_I \colon \fprod{i \in I} H_i \onto G$
is an I-cover.
Let $k \in I$
and put
$J = I \smallsetminus \{k\}$.
By Lemma~\ref{quotient of fprod not I}
we have
$\eta_I = \eta_J \circ \pr_{I,J}$
such that
$\pr_{I,J}$ is  indecomposable but not basic.
By Proposition~\ref{I is basic},
$\pr_{I,J}$ is not an I-cover.
\end{Example}

Because of condition (a) of Lemma~\ref{general}
the question arises
whether
an indecomposable
$\eta \colon H \onto G$
can be an I-cover,
if 
$H$ is the free product of $\Ker \eta$ with some subgroup of $H$.
The negative answer is provided in 
Proposition~\ref{not I-cover}
below, after some preparations.

\begin{Lemma}\label{EP for direct product}
Let $G$ be a profinite group,
$\phi \colon G \onto A$ an epimorphism onto a finite group $A$,
and $S$ a finite simple group.
Let $\alpha \colon A \times S \onto A$
be the coordinate projection.
Then the following are equivalent:
\begin{itemize}
\item[(a)]
Embedding problem $(\phi, \alpha)$ has a surjective solution.
\item[(b)]
There is
$\gamma \colon G \onto A \times S$.
\item[(c)]
There is
$\psi \colon G \onto S$
that does not factor through $\phi$.
\end{itemize}
\end{Lemma}

\begin{proof}
(a) $\implies$ (b)
is trivial.

(b) $\implies$ (c):
Let
$\gamma \colon G \onto A \times S$
be an epimorphism and denote
\begin{align*}
P &= \{ \psi \colon G \onto S\},
\\
N &= \{ \psi \colon G \onto S \st
\psi \textnormal{ factors through } \phi\},
\\
M &= \{ \psi \colon G \onto S \st
\psi \textnormal{ factors through } \gamma\},
\\
N' &= \{ \psi' \colon A \onto S\},
\\
M' &= \{ \psi' \colon A \times S \onto S\}.
\end{align*}
Then 
$M \subseteq P$ 
and $\psi' \mapsto \psi' \circ \alpha$
is an injective map
$N' \to M' \smallsetminus \{\beta\}$,
where
$\beta \colon A \times S \onto S$
is the coordinate projection.
Also,
there are bijections
$N \leftrightarrow N'$,
$M \leftrightarrow M'$.
So
$N,M,N',M'$ are finite,
and
$|N| = |N'| < |M'| = |M| \le |P|$.
Hence there is
$\psi \in P \smallsetminus N$.

(c) $\implies$ (a):
Define 
$\gamma \colon G \to A \times S$
by
$g \mapsto (\phi(g), \psi(g))$.
Then
$\alpha \circ \gamma = \phi$
and
$\beta \circ \gamma = \psi$.
We claim that $\gamma$ is surjective.
$$
\xymatrix{
G 
\arr@/^1pc/[rrd]^{\phi}
\ar[rd]^{\gamma}
\arr@/_1pc/[rdd]_{\psi}
\\
& A \times S \arr[d]^{\beta} \arr[r]_{\alpha} & A \arr[d] \dotarr[dl]^{\rho}
\\
& S \arr[r] & 1
}
$$
Indeed,
$\Ker \phi \normal G$,
so
$\psi(\Ker \phi) \normal S$.
We cannot have
$\psi(\Ker \phi) = 1$,
because
$\psi$ does not factor through $\phi$.
Thus, as $S$ is simple,
$\psi(\Ker \phi) = S$.
Therefore the square
$(G,A,S,1)$
is semi-cartesian
and 
$(A \times S,A,S,1)$
is cartesian,
hence
$\gamma$ is surjective
by Lemma~\ref{expand}.
\end{proof}

\begin{Lemma}\label{direct compact}
Consider a compact cartesian square 
\begin{equation}\label{cart prod}
\xymatrix{
G \times S \ar[r]_{\eta} \ar[d]_{\beta} & G \ar[d]^{\phi}
\\
B \ar[r]^{\alpha} & A \\
}
\end{equation}
in which
$S$ is simple
and $\eta$ is the coordinate projection.
Then, up to isomorphism,
$B = A \times S$,
$\alpha$ is the coordinate projection,
and $\beta = \phi \times \id_S$.
\end{Lemma}

\begin{proof}
Put
$A' = \beta(G)$,
$S' = \beta(S)$,
and let
$\beta_G \colon G \to A'$,
$\beta_S \colon S \to S'$
be the restrictions of $\beta$ to $S,G$, respectively.
Then $A', S' \normal B$ and $B = A' S'$.
As \eqref{cart prod} is cartesian,
$\beta$ maps $S = \Ker \eta$
isomorphically onto $\Ker \alpha$,
so $S' = \Ker \alpha$ is simple
and $\beta_S$ is an isomorphism.
Further,
$\alpha \circ \beta_G = \phi$,
so $A' \ne B$, by 
Proposition~\ref{indecomposable cartesian square}.
From this and $B = A' S'$ we get that
$S' \not\le A'$, hence, as $S'$ is simple,
$A' \cap S' = 1$.
Therefore
$B = A' \times S'$.
But $S' = \Ker \alpha$,
so $\alpha$ maps $A'$ isomorphically onto $A$.
Replacing groups by isomorphic copies
we may assume that
$A' = A$
and
$S' = S$.
Then
$\beta_G = \phi$,
$\beta_S = \id_{S}$,
$\alpha|_{A} = \id_{A}$,
and
$S = \Ker \alpha$.
\end{proof}

\begin{Remark}
The requirement that \eqref{cart prod} be compact is necessary.

Indeed,
let
$G =\langle g \rangle$
be the free pro-2 group on one generator,
$B =\langle b \rangle$
the cyclic group of order 4,
and
$ S = \langle s \rangle$ and $A =\langle a \rangle$
the cyclic groups of order 2.
Define $\phi \colon G \onto A$ by $g \mapsto a$,\ 
$\alpha \colon B \onto A$ by $b \mapsto a$,
and $\beta \colon G \times S \onto B$
by $g \mapsto b$ and $s \mapsto b^2$.
Then \eqref{cart prod} commutes
and $\beta$ maps $\Ker \eta = \langle s \rangle$
isomorphically onto
$\Ker \alpha = \langle b^2 \rangle$,
hence \eqref{cart prod} is cartesian.
But $\beta(G) \cap \beta(S) \ne 1$.
\end{Remark}

\begin{Lemma}\label{not superbasic}
Let $G$ be a profinite group
and $S$ a finite simple group.
Then the coordinate projection
$\eta \colon G \times S \onto G$
is not a superbasic cover of $G$.
\end{Lemma}

\begin{proof}
Assume
that $\eta$ is a superbasic cover.
So there is a compact cartesian square \eqref{cart prod}
with $A,B$ finite,
and an epimorphism $\delta \colon G \onto B$.
By Lemma~\ref{direct compact}
we may assume that
$B = A \times S$,
$\alpha$ is the coordinate projection,
and $\beta = \phi \times \id_S$.
By Lemma~\ref{EP for direct product},
there is an epimorphism
$\gamma \colon G \onto S \times A$
such that
$\alpha \circ \gamma = \phi$.
By Proposition~\ref{indecomposable cartesian square},
\eqref{cart prod} is not compact,
a contradiction.
\end{proof}

\begin{Proposition}\label{not I-cover}
Let $G$ be a profinite group
and $S$ a finite simple group.
Then the coordinate projection
$\eta_0 \colon G \times S \onto G$
is not an I-cover of $G$.
\end{Proposition}

\begin{proof}
Let $\eps \colon E \onto G$ be
the smallest embedding cover of $G$
of Construction~\ref{sec transfinite}.
(Thus $G_0 = G$ and $G_\mu = E$.)
For every $\lambda \le \mu$ let
$\eta_\lambda \colon G_\lambda \times S \to G_\lambda$
be the projection on $S$.
It is indecomposable.
For all
$0 \le \kappa < \lambda \le \mu$
the following diagram is cartesian
\begin{equation}\label{lambda kappa 2}
\xymatrix{
G_\lambda \times S \arr[r]_{\eta_{\lambda}}
\arr[d]_{\pi_{\lambda, \kappa} \times \id_S}
 & G_\lambda \arr[d]^{\pi_{\lambda, \kappa}}
\\
G_\kappa \times S \arr[r]^{\eta_{\kappa}} & G_\kappa 
}
\end{equation}
We have to show that
there is no
$\gamma \colon G_\mu \onto G_0 \times S$ 
such that
$\eta_0 \circ \gamma = \pi_{\mu,0}$,
that is,
by Proposition~\ref{indecomposable cartesian square},
that \eqref{lambda kappa 2}
with $\kappa = 0$ and $\lambda = \mu$,
is compact.

In fact,
we claim that
that \eqref{lambda kappa 2}
is compact
for all $\kappa < \lambda$.
In view of
Remark~\ref{square trivialities}(b)
and 
\cite[Lemma~2.15]{fund}
it suffices to show
that \eqref{lambda kappa 2}
is compact for
$\kappa = \lambda - 1$.

If it is not,
there is
$\gamma_{\lambda, \kappa} \colon G_\lambda \onto G_\kappa \times S$
such that
$\eta_k \circ \gamma_{\lambda, \kappa} = \pi_{\lambda,\kappa}$.
As $\pi_{\lambda,\kappa}$ is superbasic, and hence indecomposable,
and $\eta_\kappa$ is not an isomorphism,
$\gamma_{\lambda, \kappa}$ is an isomorphism.
Thus $\eta_k$ is superbasic.
But this is a contradiction to Lemma~\ref{not superbasic}.
\end{proof}


\begin{thebibliography}{BS99}


\bibitem[Ch]{Ch}
Z.~Chatzidakis,
\emph{Model Theory of profinite groups having the Iwasawa property},
Illinois Journal of Mathematics \textbf{42}, 70--96 (1998).

\bibitem[FH]{FH}
S.~Fried and D.~Haran,
\emph{Quasi-formations},
Israel Journal of Mathematics \textbf{229}, 193--217 (2019).


\bibitem[H]{fund}
D.~Haran,
\emph{Fundaments of epimorphisms of profinite groups},
arXiv:2509.11296 [math.GR].


\bibitem[FJ]{FJ}
M.D.~Fried and M.~Jarden,
\emph{Field Arithmetic},
Ergebnisse der Mathematik III \textbf{11},
3rd edition, revised by M. Jarden.
Springer, 2008.

\end{thebibliography}
\end{document}